\documentclass[a4paper,10pt]{amsart}

\title[Trace ideals for FIO with non-smooth
symbols]{Trace ideals for Fourier integral operators with
non-smooth symbols III}
\usepackage{cite}
\usepackage{amssymb}
\usepackage{latexsym}
\usepackage{amsmath}
\usepackage{mathrsfs}
\usepackage[X2,T1]{fontenc}
\usepackage{amsthm}
\usepackage{upgreek}
\usepackage{calc}                         

\newcommand{\scal}[2]{\langle #1,#2\rangle}

\newcommand{\rr}[1]{\mathbf R^{#1}}

\newcommand{\essup}[1]{\underset{#1}{\operatorname {ess\, sup}}}

\newcommand{\nm}[2]{\Vert #1\Vert _{#2}}
\newcommand{\nmm}[1]{\Vert #1\Vert }
\newcommand{\abp}[1]{\vert #1\vert}

\newcommand{\op}{\operatorname{Op}}
\newcommand{\opp}{\operatorname{\mathsf{Op}}}

\newcommand{\sets}[2]{\{ \, #1\, ;\, #2\, \} }
\newcommand{\ep}{\varepsilon}
\newcommand{\fy}{\varphi}

\newcommand{\cdo}{\, \cdot \, }

\newcommand{\supp}{\operatorname{supp}}

\newcommand{\eabs}[1]{\langle #1\rangle}

\newcommand{\vrum}{\vspace{0.1cm}}

\numberwithin{equation}{section}          
\newtheorem{thm}{Theorem}
\numberwithin{thm}{section}
\newtheorem*{tom}{\rubrik}
\newcommand{\rubrik}{}
\newtheorem{prop}[thm]{Proposition}
\newtheorem{cor}[thm]{Corollary}
\newtheorem{lemma}[thm]{Lemma}

\theoremstyle{definition}

\newtheorem{defn}[thm]{Definition}

\theoremstyle{remark}

\newtheorem{rem}[thm]{Remark}              

\author{Joachim Toft}

\address{Department of Mathematics and Systems Engineering,
V{\"a}xj{\"o} University, Sweden}

\email{joachim.toft@vxu.se}

\author{Francesco Concetti}

\address{Department of Mathematics,
Turin University, Italy}

\email{francesco.concetti@unito.it}

\author{Gianluca Garello}

\address{Department of Mathematics,
Turin University, Italy}

\email{gianluca.garello@unito.it}

\begin{document}

\begin{abstract}
We consider Fourier integral operators with symbols in modulation
spaces and non-smooth phase functions whose second orders of
derivatives belong to certain types of modulation space. We establish
continuity and Schatten-von Neumann properties of such operators when
acting on modulation spaces.
\end{abstract}

\maketitle

\section{Introduction}\label{sec0}

\par

The aim of this paper is to investigate Fourier integral operators
with non-smooth amplitudes (or symbols), when acting on (general, or weighted)
modulation spaces.  The amplitudes are assumed to belong to
appropriate modulation spaces, or more generally, appropriate coorbit
spaces of modulation space type. Since these coorbit spaces contains
spaces of smooth functions which belongs to certain mixed Lebesgue
spaces, together with all their derivatives, it follows that we are
able to obtain results for certain Fourier integral operator with
smooth symbols. It is assumed that the phase functions are continuous
functions with second orders of derivatives belonging to appropriate
modulation spaces (i.{\,}e. weighted "Sj{\"o}strand classes") and
satisfying appropriate non-degeneracy conditions. For such Fourier
integral operators we investigate continuity and compactness
properties when acting on modulation spaces. Especially we are
concerned with detailed compactness investigations of such operators
in background of  Schatten-von Neumann theory, when acting on Hilbert
modulation spaces. (Here we recall that the spaces of trace-class or
Hilbert-Schmidt operators are particular classes of Schatten-von
Neumann classes.) More precisely, we establish sufficient  conditions
on the amplitudes and phase functions in order for the corresponding
Fourier integral operators to be  Schatten-von Neumann of certain
degree. Since Sobolev spaces of Hilbert type are special cases of
these Hilbert modulation spaces, it follows that our results can be
applied on certain problems involving such spaces.

\par

Furthermore, by letting the involved weight functions be trivially
equal to one, these Sobolev spaces are equal to $L^2$, and our results
here then generalize those in \cite{CT1,CT2}, where similar questions
are discussed when the amplitudes and second order of derivatives
belong to classical or non-weighted modulation spaces.

\par

On the other hand, the framework of the investigations in the present
paper as well as in \cite{CT1,CT2} is to follow some ideas by
A. Boulkhemair in \cite{Bu1} and localize the Fourier integral
operators in terms of short-time Fourier transforms, and then making
appropriate Taylor expansions and estimates. In fact, in \cite{Bu1},
Boulkhemair considers a certain class of Fourier
integral operators were the corresponding symbols are defined without
any explicit regularity assumptions and with only small
regularity assumptions on the phase functions. The symbol class here
is, in the present paper, denoted by $M^{\infty ,1}$, is sometimes
called the "Sj{\"o}strand class", and
contains $S^0_{0,0}$, the set of smooth functions which are bounded
together with all their derivatives. In time-frequency
community, $M^{p,q}$ is known as a (classical or non-weighted)
modulation space with exponents $p\in [1,\infty ]$ and $q\in [1,\infty
]$. (See e.{\,}g. \cite {Gc2, Fe4, FG1} or below for strict
definition.) Boulkhemair
then proves that such operators are uniquely extendable to continuous
operators on $L^2$. In particular it follows that pseudo-differential
operators with symbols in $M^{\infty ,1}$ are $L^2$-continuous, which
was proved by J. Sj{\"o}strand in \cite {Sj1}, where it seems that
$M^{\infty ,1}$ was used for the first time in this context.

\par

Boulkhemair's result was extended in \cite{CT1, CT2}, where it is
proved that if the amplitude belongs to the classical modulation space
$M^{p,1}$, then the corresponding Fourier integral operator is
Schatten-von Neumann operator of order $p\in [1,\infty ]$ on $L^2$. In
\cite{CT2} it is also proved that if the amplitude only depends on the
phase space variables and belongs to $M^{p,p}$, then the corresponding
Fourier integral operator is continuous from $M^{p',p'}$ to
$M^{p,p}$. If in addition $1\le p\le 2$, then it is also proved that
the operator is Schatten-von Neumann of order $p$ on $L^2$.

\par

We remark that the assumptions on the phase functions imply that these
functions are two times continuously differentiable, which is usually
violated for "classical" Fourier integral operators (see
e.{\,}g. \cite {H,RuS1,RuS2,RuS3}). For example, this condition is not
fulfilled in general when the phase function is homogenous of degree
one in the frequency variable. We refer to \cite{RuS1,RuS2,RuS3} for
recent contribution to the theory of Fourier integral operators
with non-smooth symbols, and in certain domains small regularity
assumptions of the phase functions.

\par

In order to be more specific we recall some definitions. Assume that
$p,q\in [1,\infty ]$ and that $\omega \in \mathscr P(\rr {2n})$ (see
Section \ref{sec1} for the definition of $\mathscr P(\rr n)$). Then
the \emph{modulation space} $M^{p,q}_{(\omega )}(\rr
n)$ is the set of all $f\in \mathscr S'(\rr n)$ such that
\begin{equation}\label{modnorm}
\nm {f}{M^{p,q}_{(\omega )}} \equiv \Big( \int _{\rr n}\Big( \int 
_{\rr n} |\mathscr F(f \tau_x\chi)(\xi)\omega (x,\xi )|^p\,
dx\Big)^{q/p}\, d\xi \Big)^{1/q}<\infty
\end{equation}
(with obvious modification when $p=\infty$ or
$q=\infty$). Here $\tau _x$ is the translation operator $\tau _x\chi
(y)=\chi (y-x)$, $\mathscr F$ is the Fourier transform on $\mathscr
S'(\rr n)$ which is given by
$$
\mathscr Ff(\xi ) = \widehat f(\xi )\equiv (2\pi )^{-n/2}\int _{\rr n}
f(x)e^{-i\scal x\xi}\, dx
$$
when $f\in \mathscr S(\rr n)$, and $\chi \in \mathscr S(\rr
n)\setminus 0$ is called a \emph{window function} which is kept
fixed. For conveniency we set $M^{p,q}=M^{p,q}_{(\omega )}$ when
$\omega =1$.

\par

Modulation spaces were introduced by H. Feichtinger in
\cite{Fe4}. 
The basic theory of such spaces were thereafter extended
by
Feichtinger and Gr{\"o}chenig in \cite{FG1,FG2}, where the coorbit
space theory was established. Here we note that the amplitude classes
in the present paper consist of coorbit spaces, defined in such way
that their norms are given by \eqref{modnorm}, after replacing the
$L^p$ and $L^q$ norms by \emph{mixed} Lebesgue norms and interchanging
the order of integration. (See Subsection \ref{ssec1.2} and Section
\ref{sec2}.) During the last twenty years, modulation spaces have been
an active fields of research (see e.{\,}g. \cite {Gc2, Fe4, Fe5, PT1,
Te, To7}). They are rather similar to Besov spaces (see
\cite{BaC,ST,To7} for sharp embeddings) and it has appeared that they
are useful to have in background in time-frequency analysis and to
some extent also in pseudo-differential calculus.

\par

Next we discuss the definition of Fourier integral operators. For
conveniency we restricts ourself to operators which belong to
$\mathscr L(\mathscr S(\rr {n_1}),\mathscr S'(\rr {n_2}))$. Here we let
$\mathscr L(V_1,V_2)$ denote the set of all linear and continuous
operators from $V_1$ to $V_2$, when $V_1$ and $V_2$ are topological
vector spaces. For any appropriate $a \in \mathscr S'(\rr {N+m})$ (the
\emph{symbol} or \emph{amplitude}) for $N=n_1+n_2$, and real-valued
$\fy \in C(\rr {N+m})$ (the \emph{phase function}), the Fourier
integral operator $\op _\fy (a)$ is defined by the formula
\begin{equation}\label{fintop}
\op _\fy (a)f(x) =(2\pi )^{-N/2}\iint _{\rr {n_1+m}}a(x,y,\xi
)f(y)e^{i\fy (x,y,\xi )}\, dyd\xi ,
\end{equation}
when $f\in \mathscr S(\rr {n_1})$. Here the integrals should be
interpreted in distribution sense if necessary. By letting
$m=n_1=n_2=n$, and choosing symbols and phase functions in appropriate
ways, it follows that the pseudo-differential operator
$$
\op (a)f(x) =(2\pi )^{-n}\iint _{\rr n}a(x,y,\xi )f(y)e^{i\scal{x-y}\xi}
\, dyd\xi 
$$
is a special case of Fourier integral operators. Furthermore, if
$t\in \mathbf R$ is fixed, and $a$ is an appropriate function or
distribution on $\rr {2n}$ instead of $\rr {3n}$, then the definition
of the latter pseudo-differential operators cover the definition of
pseudo-differential operators of the form
\begin{equation}\label{e0.5}
a_t(x,D)f(x) = (2\pi )^{-n}\iint _{\rr {2n}}a((1-t)x+ty,\xi
)f(y)e^{i\scal{x-y}\xi}\, dyd\xi .
\end{equation}

\par

On the other hand, in the framework of harmonic analysis it follows
that the map $a\mapsto a_t(x,D)$ from $\mathscr S(\rr {2n})$ to
$\mathscr L(\mathscr S(\rr n),\mathscr S'(\rr n))$ is uniquely
extendable to a bijection from $\mathscr S'(\rr {2n})$ to $\mathscr
L(\mathscr S(\rr n),\mathscr S'(\rr n))$.

\par

In the litterature it is usually assumed that $a$ and $\fy$ in
\eqref{fintop} are smooth functions. For example, if $n_1=n_2=n$,
$a\in \mathscr S(\rr {2n+m})$ and $\fy \in
C^\infty (\rr {2n+m})$ satisfy $\fy ^{(\alpha )}\in S^0_{0,0}(\rr
{2n+m})$ when $|\alpha |=N_1$ for some integer $N_1\ge 0$, then it is
easily seen that $\op _\fy (a)$ is continuous on $\mathscr S(\rr n)$
and extends to a continuous map from $\mathscr S'(\rr n)$ to $\mathscr
S(\rr n)$. In \cite{AF} it is proved that if $\fy ^{(\alpha )}\in
S^0_{0,0}(\rr {2n+m})$ for all multli-indices $\alpha$ such that
$|\alpha |=2$ and satisfies
\begin{equation}\label{detphicond}
\left | \det \left ( \begin{matrix}
                     \fy ''_{x,y} & & \fy ''_{x,\xi}
\\[1ex]
                     \fy ''_{y,\xi } &  &\fy ''_{\xi ,\xi}
                     \end{matrix}
\right ) \right | \ge \mathsf d
\end{equation}
for some $\mathsf d >0$, then the definition of $\op _\fy$ extends
uniquely to any $a\in S^0_{0,0}(\rr {2n+m})$, and then $\op _\fy (a)$
is continuous on $L^2(\rr n)$. Next assume that $\fy$ instead
satisfies $\fy ^{(\alpha )}\in M^{\infty ,1}(\rr {3n})$ for all
multi-indices $\alpha$ such that
$|\alpha |=2$ and that \eqref{detphicond} holds for some $\mathsf d
>0$. This implies that the condition on $\fy$ is relaxed since
$S^0_{0,0} \subseteq M^{\infty ,1}$. Then Boulkhemair improves the
result in \cite{AF} by proving that the definition of $\op _\fy$
extends uniquely to any $a\in M^{\infty ,1}(\rr {2n+m})$, and that
$\op _\fy (a)$ is still continuous on $L^2(\rr n)$.

\par

In Section \ref{sec2} we discuss continuity and Schatten-von Neumann
properties for Fourier integral operators which are related to those
which were considered by Boulkhemair. More precisely, assume that
$\omega$, $\omega _j$ for $j=1,2$ and $v$ are appropriate weight
functions, $\fy \in M^{\infty ,1}_{(v)}$ and $1<p<\infty$. Then we
prove in Subection \ref{ssec2.4} that the definition of $a\mapsto \op
_\fy (a)$ from $\mathscr S$ to $\mathscr L(\mathscr S(\rr n),\mathscr
S'(\rr n)$ extends uniquely to any $a\in M^{\infty ,1}_{(\omega )}$,
and that $\op _\fy (a)$ is continuous from $M^p_{(\omega _1)}$ to
$M^{p}_{(\omega _2)}$. In particular we recover Boulkhemair's result
by letting $\omega =\omega _j=v=1$ and $p=2$.

\par

In Subsection \ref{ssec2.5} we connsider more general Fourier integral
operators, where we assume that the amplitudes belong to coorbit
spaces which, roughly speaking, are like $M^{p,q}_{(\omega )}$ for
$p,q\in [1,\infty ]$ in certain variables and like $M^{\infty
,1}_{(\omega )}$ in the other variables. (Note here that $M^{\infty
,1}_{(\omega )}$ is contained in $M^{\infty ,q}_{(\omega )}$ in view
of Proposition \ref{p1.4}.) If $q\le p$, then we prove that such
Fourier integral operators are continuous from $M^{p',p'}_{(\omega
_1)}$ to $M^{p,p}_{(\omega _2)}$. Furthermore, by interpolation
between the latter result and our extension of Boulkhemair's result we
prove that if $q\le \min (p,p')$, then these Fourier integral
operators belong to $\mathscr I_p(M^{2,2}_{(\omega
_1)},M^{2,2}_{(\omega _2)})$. Here $\mathscr I_p(\mathscr H_1,\mathscr
H_2)$ denotes the set of Schatten-von Neumann operators from the
Hilbert space $\mathscr H_1$ to $\mathscr H_2$ of order $p$. This
means that $T\in \mathscr I_p(\mathscr H_1,\mathscr H_2)$ if and only
if $T$ linear and continuous operators $T$ from $\mathscr H_1$ to
$\mathscr H_2$ which satisfy
$$
\nm T{\mathscr I_p}\equiv \sup \Big ( \sum|(Tf_j,g_j)_{\mathscr
H_2}|^p\Big ) ^{1/p}<\infty ,
$$
where the supremum should be taken over all orthonormal
sequences $(f_j)$ in $\mathscr H_1$ and $(g_j)$ in $\mathscr H_2$.

\par

In Section \ref{sec3} we list some consequences of our general results
in Section \ref{sec2}. For example, assume that $p,q\in [1,\infty ]$,
$a(x,y,\xi )=b(x,\xi )$, for some $b\in M^{p,q}_{(\omega )}(\rr
{2n})$, and that
\begin{equation}\label{detphicond2}
|\det (\fy ''_{y,\xi})|\ge \mathsf d
\end{equation}
holds for some constant $\mathsf d >0$. Then it follows from the
results in Section \ref{sec2} that if $q=p$, then $\op _\fy (a)$ is
continuous from $M^{p',p'}_{(\omega _1)}$ to $M^{p,p}_{(\omega
_2)}$. Furthermore, if in addition \eqref{detphicond} and $q\le \min
(p,p')$ hold, then $\op _\fy (a)\in \mathscr I_p$.

\par

\section{Preliminaries}\label{sec1}

\par

In this section we discuss basic properties for modulation
spaces. The proofs are in many cases omitted since they can
be found in \cite {Fe2, Fe3, Fe4, Fe5, FG1, FG2, FG4, Gc2,
To1, To2, To7}.

We start by discussing some notations. The duality between a
topological vector space and its dual is denoted by $\scal \cdo \cdo
$. For admissible $a$ and $b$ in $\mathscr S'(\rr n)$, we set
$(a,b)=\scal a{\overline b}$, and it is obvious that $(\cdo ,\cdo )$
on $L^2$ is the usual scalar product.

\par

Assume that $\mathscr B_1$ and $\mathscr B_2$ are topological
spaces. Then $\mathscr B_1\hookrightarrow \mathscr B_2$ means that
$\mathscr B_1$ is continuously embedded in $\mathscr B_2$. In the case
that $\mathscr B_1$ and $\mathscr B_2$ are Banach spaces, $\mathscr
B_1\hookrightarrow \mathscr B_2$ is equivalent to $\mathscr
B_1\subseteq \mathscr B_2$ and $\nm x{\mathscr B_2}\le C\nm x{\mathscr
B_1}$, for some constant $C>0$ which is independent of $x\in \mathscr
B_1$.

\medspace

Let $\omega, v\in L^\infty _{loc}(\rr {n})$ be positive
functions. Then $\omega$ is called $v$-\textit{moderate} if
\begin{equation}\label{moderate}
\omega(x+y) \leq C \omega(x) v(y), \quad x,y \in \rr {n} ,
\end{equation}
for some constant $C>0$, and if $v$ in \eqref{moderate} can be chosen
as a polynomial, then $\omega$ is called polynomially
moderated. Furthermore, $v$ is called \emph{submultiplicative} if
\eqref{moderate} holds for $\omega =v$. We denote by
$\mathscr P(\rr {n})$ the set of all polynomially moderated
functions on $\rr {n}$.

\par

\subsection{Modulation spaces}\label{ssec1.1}
Next we recall some properties on modulation spaces. We remark  that
the definition of modulation spaces $M^{p,q}_{(\omega )}(\rr n)$,
given in \eqref{modnorm} for $p,q\in[1, \infty]$, is independent of
the choice of the window $\chi\in \mathscr S(\mathbb R^n)\setminus 0$.
Moreover different choices of $\chi$ give rise to equivalent
norms. (See Proposition \ref{p1.4} below).

\par

For conveniency we set $M^p _{(\omega )}= M^{p,p}_{(\omega
)}$. Furthermore we  set $M^{p,q}=M^{p,q}_{(\omega )}$ if $\omega
\equiv 1$.

\par

The proof of the following proposition is omitted, since the results
can be found in \cite {Fe2, Fe3, Fe4, Fe5, FG1, FG2, FG4, Gc2,To1,
To2, To7}. Here and in what follows, $p'\in[1,\infty]$ denotes the
conjugate exponent of $p\in[1,\infty]$, i.{\,}e. $1/p+1/p'=1$ should
be fulfilled.

\par

\begin{prop}\label{p1.4}
Assume that $p,q,p_j,q_j\in [1,\infty ]$ for $j=1,2$, and $\omega
,\omega _1,\omega _2,v\in \mathscr P(\rr {2n})$ are such that $\omega$
is $v$-moderate and $\omega _2\le C\omega _1$ for some constant
$C>0$. Then the following are true:
\begin{enumerate}
\item[{\rm{(1)}}] if $\chi \in M^1_{(v)}(\rr n)\setminus 0$, then $f\in
M^{p,q}_{(\omega )}(\rr n)$ if and only if \eqref {modnorm} holds,
i.{\,}e. $M^{p,q}_{(\omega )}(\rr n)$ is independent of the choice of
$\chi$. Moreover, $M^{p,q}_{(\omega )}$ is a Banach space under the
norm in \eqref{modnorm}, and different choices of $\chi$ give rise to
equivalent norms;

\vrum

\item[{\rm{(2)}}] if  $p_1\le p_2$ and $q_1\le q_2$  then
$$
\mathscr S(\rr n)\hookrightarrow M^{p_1,q_1}_{(\omega _1)}(\rr
n)\hookrightarrow M^{p_2,q_2}_{(\omega _2)}(\rr n)\hookrightarrow
\mathscr S'(\rr n)\text ;
$$

\vrum

\item[{\rm{(3)}}] the $L^2$ product $( \cdo ,\cdo )$ on $\mathscr
S$ extends to a continuous map from $M^{p,q}_{(\omega )}(\rr
n)\times M^{p'\! ,q'}_{(1/\omega )}(\rr n)$ to $\mathbf C$. On the
other hand, if $\nmm a = \sup \abp {(a,b)}$, where the supremum is
taken over all $b\in \mathscr {S}(\rr n)$ such that
$\nm b{M^{p',q'}_{(1/\omega )}}\le 1$, then $\nmm {\cdot}$ and $\nm
\cdot {M^{p,q}_{(\omega )}}$ are equivalent norms;

\vrum

\item[{\rm{(4)}}] if $p,q<\infty$, then $\mathscr S(\rr n)$ is dense in
$M^{p,q}_{(\omega )}(\rr n)$. The dual space of $M^{p,q}_{(\omega
)}(\rr n)$ can be identified
with $M^{p'\! ,q'}_{(1/\omega )}(\rr n)$, through the form $(\cdo  ,\cdo
)_{L^2}$. Moreover, $\mathscr S(\rr n)$ is weakly dense in $M^{\infty
}_{(\omega )}(\rr n)$.
\end{enumerate}
\end{prop}

\par

Proposition \ref{p1.4}{\,}(1) allows us  be rather vague concerning
the choice of $\chi \in  M^1_{(v)}\setminus 0$ in
\eqref{modnorm}. For example, if $C>0$ is a constant and $\mathscr A$ is a
subset of $\mathscr S'$, then $\nm a{M^{p,q}_{(\omega )}}\le C$ for
every $a\in \mathscr A$, means that the inequality holds for some choice
of $\chi \in  M^1_{(v)}\setminus 0$ and every $a\in
\mathscr A$. Evidently, a similar inequality is true for any other choice
of $\chi \in  M^1_{(v)}\setminus 0$, with  a suitable constant, larger
than $C$ if necessary.

\par

It is also convenient to let $\mathcal M^{p,q}_{(\omega )}(\rr {n})$
be the completion of $\mathscr S(\rr n)$
under the norm $\nm \cdot {M^{p,q}_{(\omega )}}$. Then $\mathcal
M^{p,q}_{(\omega )}\subseteq M^{p,q}_{(\omega )}$ with equality if and
only if $p<\infty$ and $q<\infty$. It follows that most of the
properties which are valid for $M^{p,q}_{(\omega )}(\rr n)$, also hold
for $\mathcal M^{p,q}_{(\omega )}(\rr n)$.

\par

We also need to use multiplication properties of modulation
spaces. The proof of the following proposition is omitted since the
result can be found in \cite {Fe4, FG1, To2,To7}.

\par

\begin{prop}\label{multprop}
Assume that $p, p_j,q_j\in [1,\infty ]$ and $\omega _j, v\in \mathscr
P(\rr {2n})$ for $j=0,\dots ,N$ satisfy
$$
\frac 1{p_1}+\cdots +\frac 1{p_N}=\frac 1{p_0},\quad \frac
1{q_1}+\cdots +\frac 1{q_N}=N-1+\frac 1{q_0},
$$
and
$$
\omega _0(x,\xi _1+\cdots +\xi _N)\le C\omega _1(x,\xi _1)\cdots
\omega _N(x,\xi _N),\quad x,\xi _1,\dots \xi _N \in \rr n ,
$$
for some constant $C$. Then $(f_1,\dots ,f_N)\mapsto
f_1\cdots f_N$ from $\mathscr S(\rr n)\times \cdots \times \mathscr
S(\rr n)$ to $\mathscr S(\rr n)$ extends uniquely to a continuous map
from $M^{p_1,q_1}_{(\omega _1)}(\rr n)\times \cdots \times
M^{p_N,q_N}_{(\omega _N)}(\rr n)$ to $M^{p_0,q_0}_{(\omega _0)}(\rr
n)$, and
$$
\nm {f_1\cdots f_N}{M^{p_0,q_0}_{(\omega _0)}}\le C\nm
{f_1}{M^{p_1,q_1}_{(\omega _1)}}\cdots \nm
{f_N}{M^{p_N,q_N}_{(\omega _N)}}
$$
for some constant $C$ which is independent of $f_j\in
M^{p_j,q_j}_{(\omega _j)}(\rr n)$ for $i=1,\dots ,N$.

\par

Furthermore, if  $u_0=0$ when $p<\infty$, $v(x,\xi )=v(\xi )\in
\mathscr P(\rr n)$ is submultiplicative, $f\in M^{p ,1}_{(v)}(\rr n)$,
and $\phi ,\psi$ are entire funcitons on $\mathbf C$ with expansions
$$
\phi (z)=\sum _{k=0}^\infty u_kz^k,\quad \psi (z)=\sum _{k=0}^\infty
|u_k|z^k,
$$
then $\phi (f)\in M^{p ,1}_{(v)}(\rr n)$, and
$$
\nm {\phi (f)}{M^{p ,1}_{(v)}}\le C\, \psi (C\nm f{M^{p,1}_{(v)}}),
$$
for some constant $C$ which is independent of $f\in M^{p,1}_{(v)}(\rr
n)$
\end{prop}

\par

\begin{rem}\label{p1.7}
Assume that $p,q,q_1,q_2\in [1,\infty ]$, $\omega _1\in \mathscr P(\rr
n)$ and that $\omega, v \in \mathscr P(\rr {2n})$ are such that
$\omega$ is $v$-moderate. Then the following properties for modulation
spaces hold:
\begin{enumerate}
\item[{\rm{(1)}}]  if $q_1\le\min (p,p')$, $q_2\ge \max (p,p')$ and
$\omega (x,\xi )=\omega _1(x)$,
then $M^{p,q_1}_{(\omega )}\subseteq L^p_{(\omega _0)}\subseteq
M^{p,q_2}_{(\omega )}$. In particular,
$M^2_{(\omega )}=L^2_{(\omega _0)}$;

\vrum

\item[{\rm{(2)}}] if $\omega (x,\xi )=\omega _1(x)$, then
$M^{p,q}_{(\omega )}(\rr n)\hookrightarrow C(\rr n)$ if and only if
$q=1$;

\vrum

\item[{\rm{(3)}}] $M^{1,\infty}$ is a convolution algebra which
contains all measures on $\rr n$ with bounded mass;

\vrum

\item[{\rm{(4)}}] if $x_0\in \rr n$ and $\omega _0(\xi )=\omega
(x_0,\xi )$, then $M^{p,q}_{(\omega )}\cap \mathscr E' =\mathscr
FL^q_{(\omega _0)}\cap \mathscr E'$. Here $FL^q_{(\omega _0)}(\rr n)$
consists of all $f\in \mathscr S'(\rr n)$ such that
$$
\nm {\widehat f \, \omega _0}{L^q}<\infty .
$$
Furthermore, if $B$ is a ball
with radius $r$ and center at $x_0$, then
$$
C^{-1}\nm {\widehat f}{L^q_{(\omega _0)}}\le \nm
f{M^{p,q}_{(\omega )}}\le C\nm {\widehat f}{L^q_{(\omega
_0)}},\quad f\in \mathscr E'(B)
$$
for some constant $C$ which only depends on $r$, $n$, $\omega $ and
the chosen window functions;

\vrum

\item[{\rm{(5)}}] if $\omega (x,\xi )=\omega (\xi ,x)$, then
$M^{p}_{(\omega )}$ is invariant under the Fourier transform. A
similar fact holds for partial Fourier transforms;

\vrum

\item[{\rm{(6)}}] for each $x,\xi\in \rr n$ we have
$$
\Vert e^{i\scal\cdo \xi}f(\cdo -x)   \Vert_{M^{p,q}_{(\omega)}}\le C
v(x,\xi) \Vert f \Vert_{M^{p,q}_{(\omega)}},
$$
for some constant $C$;

\vrum

\item[{\rm{(7)}}] if $\tilde {\omega}(x,\xi)=\omega(x,-\xi)$ then
$f\in M^{p,q}_{(\omega)}$ if and only if $\bar f\in
M^{p,q}_{(\tilde\omega)}$.

\end{enumerate}
(See e.{\,}g. \cite {Fe2, Fe3, Fe4, FG1, FG2, FG4, Gc2, To7}.)
\end{rem}

\par

For future references we note that the constant $C_{r,n}$ is
independent of the center of the ball $B$ in (4) in Remark
\ref{p1.7}.

\par

In our investigations we need the following characterization of
modulation spaces.

\par

\begin{prop}\label{modspchar}
Let $\{ x_\alpha \} _{\alpha \in I}$ be a lattice in $\rr n$,
$B_\alpha =x_\alpha +B$ where $B\subseteq \rr n$ is an open ball, and
assume that $f_\alpha \in \mathscr E'(B_\alpha )$ for every $\alpha
\in I$. Also assume that $p,q\in [1,\infty ]$. Then the following is
true:
\begin{enumerate}
\item if 
\begin{equation}\label{fsum}
f = \sum _{\alpha \in I}f_\alpha \quad \text{and}
\quad
F(\xi ) \equiv \Big (\sum _{\alpha \in i}|\widehat f_\alpha (\xi
)\omega (x_\alpha ,\xi )|^p\Big )^{1/p} \in L^q(\rr n),
\end{equation}
then $f\in M^{p,q}_{(\omega )}$, and $f\mapsto \nm F{L^q}$ defines a
norm on $M^{p,q}_{(\omega )}$ which is equivalent to $\nm \cdo
{M^{p,q}_{(\omega )}}$ in \eqref{modnorm};

\vrum

\item if in addition $\cup _\alpha B_\alpha =\rr n$, $\chi \in C^\infty
_0(B)$ satisfies $\sum _\alpha \chi (\cdo -x_\alpha )=1$, $f\in
M^{p,q}_{(\omega )}(\rr n)$, and $f_\alpha =f\, \chi (\cdo -x_\alpha
)$, then $f_\alpha \in \mathscr E'(B_\alpha )$ and \eqref{fsum} is
fulfilled.
\end{enumerate}
\end{prop}

\par

\begin{proof}
(1) Assume that $\chi \in C_0^\infty (\rr n)\setminus 0$ is
fixed. Since there is a bound of overlapping supports of $f_\alpha$,
we obtain
\begin{multline*}
|\mathscr F(f\chi (\cdo -x))(\xi )\omega (x,\xi )| \le \sum  |\mathscr
F(f_\alpha \chi (\cdo -x ))(\xi )\omega (x,\xi )|
\\[1ex]
\le C \Big (\sum |\mathscr F(f_\alpha \chi (\cdo -x ))(\xi
)\omega (x,\xi )|^p\Big )^{1/p}.
\end{multline*}
for some constant $C$. From the support properties of $\chi$, and the
fact that $\omega $ is $v$-moderate for some $v\in \mathscr P(\rr
{2n})$, it follows for some constant $C$ independent of $\alpha$ we
have
$$
|\mathscr F(f_\alpha \chi (\cdo -x))(\xi )\omega (x,\xi )| \le
C|\mathscr F(f_\alpha \chi (\cdo -x))(\xi )\omega (x_\alpha ,\xi )|.
$$
Hence, for some balls $B'$ and $B'_\alpha =x_\alpha +B'$, we get
\begin{multline*}
\Big (\int _{\rr n}|\mathscr F(f\chi (\cdo -x))(\xi )\omega (x,\xi )|^p \,
dx\Big )^{1/p}
\\[1ex]
\le C \Big ( \sum _\alpha \int _{B'_\alpha }|\mathscr F(f_\alpha \chi
(\cdo -x))(\xi )\omega (x_\alpha ,\xi )|^p \, dx\Big )^{1/p}
\\[1ex]
\le C \Big ( \sum _\alpha \int _{B'_\alpha }\big (|\widehat f_\alpha
\omega (x_\alpha ,\cdo )|*|\widehat  \chi v(0,\cdo )|(\xi )\big )^p \,
dx\Big )^{1/p}
\\[1ex]
\le C'' \Big ( \sum _\alpha \big (|\widehat f_\alpha \omega (x_\alpha
,\cdo )|*|\widehat  \chi v(0,\cdo )|(\xi )\big )^p \Big )^{1/p}
\le C'' F*|\widehat \chi v(0,\cdo )|(\xi ),
\end{multline*}
for some constants $C'$ and $C''$. Here we have used Minkowski's
inequality in the last inequality. By applying the $L^q$-norm and
using Young's inequality we get
$$
\nm f{M^{p,q}_{(\omega )}}\le C'' \nm {F*|\widehat \chi v(0,\cdo
)|}{L^q}\le C'' \nm F{L^q}\nm {\widehat \chi v(0,\cdo )}{L^1}.
$$
Since we have assumed that $F\in L^q$, it follows that $\nm
f{M^{p,q}_{(\omega )}}$ is finite. This proves (1).

\par

The assertion (2) follows immediately from the general theory of
modulation spaces. (See e.{\,}g. \cite {GH1, Gc2}.) The proof is
complete.
\end{proof}

\par

Next we discuss (complex) interpolation properties for modulation
spaces. Such properties were carefully investigated in
\cite {Fe4} for classical modulation spaces, and thereafter
extended in several directions in \cite{FG2}, where
interpolation properties for coorbit spaces were established. (See
also Subsection \ref{ssec1.2}.) As a consequence of \cite{FG2} we have
the following proposition.

\par

\begin{prop}\label{interpolmod}
Assume that $0<\theta <1$, $p_j,q_j\in [1,\infty ]$ and that $\omega
_j\in \mathscr P(\rr {2n})$ for $j=0,1,2$ satisfy
$$
\frac 1{p_0} = \frac {1-\theta}{p_1}+\frac {\theta}{p_2},\quad
\frac 1{q_0} = \frac {1-\theta}{q_1}+\frac {\theta}{q_2}\quad
\text{and}\quad \omega _0=\omega _1^{1-\theta }\omega _2^\theta .
$$
Then
$$
(\mathcal M^{p_1,q_1}_{(\omega _1)}(\rr n),\mathcal
M^{p_2,q_2}_{(\omega _2)}(\rr n))_{[\theta ]} = \mathcal
M^{p_0,q_0}_{(\omega _0)}(\rr n).
$$
\end{prop}

\medspace

Next we recall some facts in Section 2 in \cite{To9} on narrow
convergence. For any $f\in \mathscr S'(\rr n)$, $\omega \in \mathscr
P(\rr {2n})$, $\chi \in \mathscr S(\rr n)$ and $p\in [1,\infty ]$, we
set
$$
H_{f,\omega ,p}(\xi )=\Big ( \int _{\rr n}\abp {\mathscr F(f\, \tau _x\chi
)(\xi )\omega (x,\xi )}^p\, dx\Big )^{1/p}.
$$

\par

\begin{defn}\label{p2.1}
Assume that $f,f_j\in M^{p,q}_{(\omega )}(\rr {m} )$,
$j=1,2,\dots \ $. Then $f_j$ is said to converge {\it narrowly} to $f$
(with respect to $p,q\in [1,\infty ]$, $\chi \in \mathscr S(\rr
m)\setminus 0$ and $\omega \in \mathscr P(\rr {2m})$), if the following
conditions are satisfied:
\begin{enumerate}
\item $f_j\to f$ in $\mathscr S'(\rr {n} )$ as $j$ turns to $\infty$;

\par

\item $H_{f_j,\omega ,p}(\xi )\to H_{f,\omega ,p}(\xi )$ in $L^q(\rr
n)$ as $j$ turns to $\infty$.
\end{enumerate}
\end{defn}

\par

\begin{rem}\label{p2.2}
Assume that $f,f_1,f_2,\dots \in \mathscr S'(\rr
n)$ satisfies (1) in Definition \ref{p2.1}, and assume that $\xi \in
\rr n$. Then it follows from Fatou's lemma that
$$
\liminf _{j\to \infty}H_{f_j,\omega ,p}(\xi )\ge H_{f,\omega ,p}(\xi
)\quad \text{and}\quad \liminf _{j\to \infty}\nm
{f_j}{M^{p,q}_{(\omega )}}\ge \nm {f}{M^{p,q}_{(\omega )}}.
$$
\end{rem}

\par

The following proposition is important to us later on. We omit the
proof since the result is a restatement of Proposition 2.3 in \cite
{To9}.

\par

\begin{prop}\label{p2.3}
Assume that $p,q\in [1,\infty ]$ such
that $q<\infty$, and that $\omega \in \mathscr P(\rr {2n})$. Then
$C_0^\infty (\rr n)$ is dense in $M^{p,q}_{(\omega )}(\rr n)$ with
respect to the narrow convergence.
\end{prop}

\medspace

\subsection{Coorbit spaces of modulation space types}\label{ssec1.2}
Next we discuss a familly of Banach spaces of time-frequency type
which contains the modulation spaces. Certain types of these Banach
spaces are used as symbol classes for Fourier integral
operators which are considered in Subsection \ref{ssec2.5}. (Cf. the
introduction.)

\par

Assume that $V_j$ and $W_j$ for $j=1,\dots 4$ are vector spaces of
dimensions $n_j$ and $m_j$ respectively such that
\begin{equation}\label{dirsums}
V_1\oplus V_2=V_3\oplus V_4 =\rr n,\qquad W_1\oplus W_2=W_3\oplus W_4
=\rr m.
\end{equation}
We let the euclidean structure in $V_j$ and $W_j$ be inherited from
$\rr n$ and $\rr m$ respectively. For
conveniency we also use the notation $\overline V$ and $\mathsf p$ for
the quadruple $(V_1,\dots V_4)$ of the vector spaces $V_1,\dots ,V_4$
and the quadruple $(p,q,r,s)$ in $[1,\infty ]^4$, respectively. That
is
$$
\overline V = (V_1,\dots ,V_4),\quad \overline W = (W_1,\dots ,W_4),\quad
\text{and}\quad \mathsf p=(p,q,r,s),
$$
where $p,q,r,s\in [1,\infty ]$. We also let and $\omega \in \mathscr
P(\rr {2n})$, and we set
$$
L^{\mathsf p}(\overline V)\equiv L^s(V_4;L^r(V_3;L^q(V_2;L^p(V_1)))).
$$
Finally we let $L^{\mathsf p}_{(\omega )}(\overline V)$ be the Banach
space which consists of all $F\in L^1_{loc}(\rr {2n})$ such that
$F\omega \in L^{\mathsf p}(\overline V)$. This means that
$L^{\mathsf p}_{(\omega )}(\overline V)$ is the set of all $F\in
L^1_{loc}(\rr {2n})$ such that
$$
\nm F{L^{\mathsf p}_{(\omega )}(\overline V)}
\equiv \Big ( \int _{V_4} \Big (  \int _{V_3} \Big (
\int _{V_2} \Big (  \int _{V_1} |F(x,\xi )\omega
(x,\xi )|^{p}\, dx_1\Big )^{q/p}\, dx_2 \Big )^{r/q}\, d\xi _1 \Big
)^{s/r} \, d\xi _2\Big )^{1/s}
$$
is finite (with obvious modifications when one or more of $p,q,r,s$
are equal to infinity). Here $dx_1,dx_2,d\xi _1 ,d\xi _2$ denotes the
Lebesgue measure in $V_1,V_2,V_3, V_4$ respectively.

\par

Next let $\chi \in \mathscr S(\rr n)\setminus 0$. Then the
space $\boldsymbol \Theta _{(\omega )}^{\mathsf p}(\overline V)$
consists of all $f\in \mathscr S'(\rr n)$ such that $V_\chi f\in
L^{\mathsf p}_{(\omega )}(\overline V)$, i.{\,}e.
\begin{equation}\label{coorbnorm}
\nm f{\boldsymbol \Theta _{(\omega )}^{\mathsf p}(\overline V)}
\equiv \nm {V_\chi}{L^{\mathsf p}_{(\omega )}(\overline V)}<\infty .
\end{equation}

\par

We note that if $p=q$ and $r=s$, then $\boldsymbol \Theta _{(\omega
)}^{\mathsf p}(\overline V)$ agrees with the modulation space
$M^{p,r}_{(\omega )}$. On the other hand, if $p\neq q$ or $r\neq s$,
then $\boldsymbol \Theta _{(\omega )}^{\mathsf p}(\overline V)$ is
\emph{not} a modulation space. However it is still a \emph{coorbit
space}, a familly of Banach spaces which is introduced and briefly
investigated in \cite{FG1,FG2}. By Corollary 4.6 in \cite{FG1} it
follows that $\boldsymbol \Theta _{(\omega )}^{\mathsf p}(\overline
V)$ is homeomorphic to a retract of $L^{\mathsf p}_{(\omega
)}(\overline V)$, which implies that the interpolation
properties of $L^{\mathsf p}_{(\omega )}(\overline V)$ spaces carry
over to $\boldsymbol \Theta _{(\omega )}^{\mathsf p}(\overline V)$
spaces. (Cf. Theorem 4.7 in \cite {FG1}.) Furthermore, if $\omega$ is
the same in the involved spaces, it follows that $L^{\mathsf
p}_{(\omega )}(\overline V)$ has the same interpolation properties as
$L^{\mathsf p}(\overline V)$. From these observations together with
the fact that the proof of Theorem 5.6.3 in \cite{BL} shows that
\begin{multline*}
(L^{p_1}(\rr n;\mathscr B_1),L^{p_2}(\rr n;\mathscr B_2))_{[\theta ]}
= L^{p}(\rr n;\mathscr B),
\\[1ex]
\text{when} \quad p, p_1, p_2\in [1,\infty ], \mathscr B =(\mathscr
B_1,\mathscr B_2)_{[\theta ]}\quad  \text{and}\quad \frac 1p =\frac
{1-\theta}{p_1}+\frac \theta{p_2},
\end{multline*}
it follows that the following result is an immediate consequence of
Theorems 4.4.1 and 5.1.1 \cite{BL}. The second part is also a
consequence of Corollary 4.6 in \cite {FG1} and Theorem ?? in
\cite{LP}. Here we use the convention
$$
1/\mathsf p =(1/p,1/q,1/r,1/s)\quad \text{when}\quad \mathsf
p=(p,q,r,s).
$$

\par

\begin{prop}\label{coorbintepol}
Assume that $V_j\subseteq \rr n$ and $W_j\subseteq \rr m$ for
$j=1,\dots ,4$ be vector
spaces such that \eqref{dirsums} holds, $\mathsf p_j,\mathsf q_j\in
[1,\infty ]^4$ for $j=0,1,2$ satisfy
$$
\frac 1{\mathsf p_0}=\frac {1-\theta}{\mathsf p_1}+\frac
{\theta}{\mathsf p_2},\quad \frac 1{\mathsf q_0}=\frac
{1-\theta}{\mathsf q_1}+\frac {\theta}{\mathsf q_2},
$$
for some $\theta \in [0,1]$. Also assume that $\omega , \omega _j\in
\mathscr P(\rr {2nj})$ for $j=1,2$. Then the following is true:
\begin{enumerate}
\item  the complex interpolation space $(\boldsymbol \Theta _{(\omega
)}^{\mathsf p_1}(\overline V), \boldsymbol \Theta _{(\omega
)}^{\mathsf p_2}(\overline V))_{[\theta ]}$ is equal to $\boldsymbol
\Theta _{(\omega )}^{\mathsf p}(\overline V)$;

\vrum

\item if $T$ is a linear and continuous map from
$\boldsymbol \Theta _{(\omega _1)}^{\mathsf p_1}(\overline
V)+\boldsymbol \Theta _{(\omega _1)}^{\mathsf p_2}(\overline V)$ to
$\boldsymbol \Theta _{(\omega _2)}^{\mathsf q_1}(\overline
W)+\boldsymbol \Theta _{(\omega _2)}^{\mathsf q_2}(\overline W)$,
which restricts to a continuous mappings from $\boldsymbol \Theta
_{(\omega _1)}^{\mathsf p_j}(\overline V)$ to $\boldsymbol \Theta
_{(\omega _2)}^{\mathsf q_j}(\overline W)$ for $j=1,2$, then $T$
restricts to a continuous map from $\boldsymbol \Theta
_{(\omega _1)}^{\mathsf p_0}(\overline V)$ to $\boldsymbol \Theta
_{(\omega _2)}^{\mathsf q_0}(\overline W)$.
\end{enumerate}
\end{prop}

\par

The most of the properties for modulation spaces stated in Proposition
Proposition \ref{p1.4} and Remark \ref{p1.7} carry over to
$\boldsymbol \Theta _{(\omega )}^{\mathsf p}$ spaces. For example the
analysis in \cite{Gc2} shows that the following result holds. Here we
use the convention
\begin{alignat*}{2}
\mathsf p_1&\le \mathsf p_2&\quad \text{when}  \quad \mathsf p_j &=
(p_j,q_j,r_j,s_j)\quad \text{and}\quad p_1\le p_2,\ q_1\le q_2,\
r_1\le r_2,\ s_1\le s_2,
\intertext{and}
t_1&\le \mathsf p\le t_2 &\quad \text{when} \quad \mathsf p &=
(p,q,r,s),\ t_1,t_2\in \mathbf R\cup \{ \infty \} \quad
\text{and}\quad t_1\le p,q,r,s \le t_2 .
\end{alignat*}

\par

\begin{prop}\label{p1.4AA}
Assume that $\mathsf p,\mathsf p_j\in [1,\infty ]$ for $j=1,2$, and
$\omega ,\omega _1,\omega _2,v\in \mathscr P(\rr {2n})$ are such that
$\omega$ is $v$-moderate and $\omega _2\le C\omega _1$ for some
constant $C>0$. Then the following are true:
\begin{enumerate}
\item[{\rm{(1)}}] if $\chi \in M^1_{(v)}(\rr n)\setminus 0$, then
$f\in \boldsymbol \Theta _{(\omega )}^{\mathsf p}(\overline V)$ if and
only if \eqref {coorbnorm} holds,
i.{\,}e. $\boldsymbol \Theta _{(\omega )}^{\mathsf p}(\overline V)$ is
independent of the choice of $\chi$. Moreover, $\boldsymbol \Theta
_{(\omega )}^{\mathsf p}(\overline V)$ is a Banach space under the
norm in \eqref{coorbnorm}, and different choices of $\chi$ give rise
to equivalent norms;

\vrum

\item[{\rm{(2)}}] if  $\mathsf p_1\le \mathsf p_2$ then
$$
\mathscr S(\rr n)\hookrightarrow \boldsymbol \Theta _{(\omega
_1)}^{\mathsf p_1}(\overline V)\hookrightarrow \boldsymbol \Theta
_{(\omega _2)}^{\mathsf p_2}(\overline V)\hookrightarrow \mathscr
S'(\rr n)\text .
$$
\end{enumerate}
\end{prop}

\par

Later on we also need the following observation.

\par

\begin{prop}\label{extops}
Assume that $(x,y)\in V_1\oplus V_2=\rr {n_0+n}$ with dual variables
$(\xi ,\eta )\in V_4\oplus V_3$, where $V_1=V_4=\rr {n_0}$ and
$V_2=V_3=\rr {n}$. Also assume that $f\in \mathscr S'(\rr {n})$,
$f_0\in \mathscr S'(\rr {n_0+n})$, $\omega \in \mathscr P(\rr {2n})$
and $\omega _0\in \mathscr P(\rr {2(n_0+n)})$  satisfy
$$
f_0(x,y)=f(y)\quad \text{(in $\mathscr S'(\rr {n_0+n})$)}\quad
\text{and} \quad \omega _0(x,y,\xi ,\eta ) = \omega (y,\eta )\eabs \xi
^t
$$
for some $t\in \mathbf R$, and that $p,q\in [1,\infty ]$. Then $f\in
M^{p,q}_{(\omega )}(\rr n)$ if and only if $f_0\in M^{\mathsf
p}_{(\omega _0)}(\rr {n_0+n})$ and $\mathsf p=(\infty ,p,q,1)$.
\end{prop}

\par

\begin{proof}
Let $\chi _0=\chi _1\otimes \chi$, where $\chi _1\in \mathscr S(\rr
{n_0})$ and $\chi \in \mathscr S(\rr {n})$. By straight-forward
computations it follows that
\begin{equation}\label{STFTrel3}
|V_{\chi _0}f_0(x,y,\xi ,\eta )\omega _0(x,y,\xi ,\eta )| = |V_{\chi
}f(y,\eta )\omega (y,\eta )|\, |\widehat \chi _1(\xi )\eabs \xi ^t|.
\end{equation}
Since $|\widehat \chi _1(\xi )|\eabs \xi ^t$ is rapidly decreasing to
zero at infinity, the result follows by applying an appropriate mixed
Lebesgue norm on \eqref{STFTrel3}.
\end{proof}

\medspace

\subsection{Schatten-von Neumann classes and pseudo-differential
operators}\label{ssec1.3}
Next we recall some facts in Chapter XVIII in \cite {H} concerning
pseudo-differential operators. Assume that $a\in \mathscr 
S(\rr {2n})$, and that $t\in \mathbf R$ is fixed. Then the
pseudo-differential operator $a_t(x,D)$ in \eqref{e0.5} is a linear
and continuous operator on $\mathscr S(\rr n)$, as remarked in the
introduction. For general $a\in \mathscr S'(\rr {2n})$, the
pseudo-differential operator $a_t(x,D)$ is defined as the continuous
operator from $\mathscr S(\rr n)$ to $\mathscr S'(\rr n)$ with
distribution kernel
\begin{equation}\label{weylkernel}
K_{t,a}(x,y)=(2\pi )^{-n/2}(\mathscr F_2^{-1}a)((1-t)x+ty,y-x),
\end{equation}
Here $\mathscr F_2F$ is the partial
Fourier transform of $F(x,y)\in \mathscr S'(\rr{2n})$ with respect to
the $y$-variable. This definition makes sense, since
the mappings $\mathscr F_2$ and $F(x,y)\mapsto F((1-t)x+ty,y-x)$ are
homeomorphisms on $\mathscr S'(\rr {2n})$. Moreover, it agrees with
the operator in \eqref{e0.5} when $a\in \mathscr S(\rr {2m})$.

\par

Furthermore, for any $t\in \mathbf R$ fixed, it follows from the kernel
theorem by Schwartz that the map $a\mapsto
a_t(x,D)$ is bijective from $\mathscr S'(\rr {2n})$ to $\mathscr
L(\mathscr S(\rr n), \mathscr S'(\rr n))$ (see e.{\,}g. \cite {H}).

\par

In particular, if $a\in \mathscr S'(\rr {2m})$ and $s,t\in
\mathbf R$, then there is a unique $b\in \mathscr S'(\rr {2m})$ such
that $a_s(x,D)=b_t(x,D)$. By straight-forward applications of
Fourier's inversion  formula, it follows that
\begin{equation}\label{pseudorelation}
a_s(x,D)=b_t(x,D) \quad \Leftrightarrow \quad b(x,\xi )=e^{i(t-s)\scal
{D_x}{D_\xi}}a(x,\xi ).
\end{equation}
(Cf. Section 18.5 in \cite{H}.)

\medspace

Next we recall some facts on Schatten-von Neumann
operators and pseudo-differential operators (cf. the introduction).

\par

For each pairs of Hilbert spaces $\mathscr H_1$ and $\mathscr H_2$,
the set $\mathscr I_p(\mathscr H_1,\mathscr H_2)$ is a Banach space
which increases with $p\in [1,\infty ]$, and if $p<\infty$, then
$\mathscr I_p(\mathscr H_1,\mathscr H_2)$ is contained in
the set of compact operators. Furthermore, $\mathscr I_1(\mathscr
H_1,\mathscr H_2)$,
$\mathscr I_2(\mathscr H_1,\mathscr H_2)$ and $\mathscr I_\infty
(\mathscr H_1,\mathscr H_2)$ agree with the set of
trace-class operators, Hilbert-Schmidt operators and continuous
operators respectively, with the same norms.

\par

Next we discuss complex interpolation properties of Schatten-von
Neumann classes. Let $p,p_1,p_2\in [1,\infty ]$ and let $0\le \theta
\le 1$. Then similar \emph{complex} interpolation properties hold for
Schatten-von Neumann classes as for Lebesgue spaces, i.{\,}e. it holds
\begin{equation}\label{interpschatt}
\mathscr I_p =(\mathscr I_{p_1},\mathscr I_{p_2})_{[\theta ]},\quad
\text{when}\quad \frac 1p = \frac {1-\theta}{p_1}+\frac \theta {p_2}.
\end{equation}
(Cf. \cite {Si}.) Furthermore, by Theorem 2.c.6 in \cite{Kon} and its
proof, together with the remark which followed that theorem, it
follows that the \emph{real} interpolation property
\begin{equation}\label{realinterpschatt}
\mathscr I_p =(\mathscr I_{2},\mathscr I_{\infty })_{\theta ,p},\quad
\text{when}\quad \theta = 1-\frac 2p.
\end{equation}
We refer to
\cite{Si, To8} for a brief discussion of Schatten-von Neumann
operators.

\par

For any $t\in \mathbf R$ and $p\in [1,\infty ]$, let $s_{t,p}(\omega
_1,\omega _2)$ be the set of all $a\in \mathscr S'(\rr {2n})$ such
that $a_t(x,D)\in \mathscr I_p(M^2_{(\omega _1)},M^2_{(\omega
_2)})$. Also set
$$
\nm a{s_{t,p}}=\nm a{s_{t,p}(\omega _1,\omega _2)}\equiv \nm
{a_t(x,D)}{\mathscr I_p(M^2_{(\omega _1)},M^2_{(\omega _2)})}
$$
when $a_t(x,D)$ is continuous from $M^2_{(\omega _1)}$ to
$M^2_{(\omega _2)}$. By using the fact that $a\mapsto
a_t(x,D)$ is a bijective map from $\mathscr S'(\rr {2n})$ to $\mathscr
L(\mathscr S(\rr n),\mathscr S'(\rr n))$, it follows that the map
$a\mapsto a_t(x,D)$ restricts to an isometric bijection from
$s_{t,p}(\omega _1,\omega _2)$ to $\mathscr I_p(M^2_{(\omega
_1)},M^2_{(\omega _2)})$.

\par

Here and in what follows we let $p'\in [1,\infty ]$ denote the
conjugate exponent of $p\in [1,\infty ]$, i.{\,}e. $1/p+1/p'=1$.

\par

\begin{prop}\label{pseudomod}
Assume that $p,q_1,q_2\in [1,\infty ]$ are such that $q_1\le \min
(p,p')$ and $q_2\ge \max (p,p')$. Also assume that $\omega _1,\omega
_2\in \mathscr P(\rr {2n})$ and $\omega ,\omega _0\in \mathscr P(\rr
{4n})$ satisfy
$$
\frac{\omega _2(x-ty,\xi + (1-t)\eta )}{\omega _1(x+(1-t)y,\xi -t\eta
)}=\omega (x,\xi ,\eta ,y)
$$
and
$$
\omega_0(x,y,\xi,\eta)=\omega((1-t)x+ty, t\xi-(1-t)\eta,\xi+\eta,y-x).
$$
Then the following is true:
\begin{enumerate}
\item $M^{p,q_1}_{(\omega)}(\rr {2n})\subseteq
s_{t,p}(\omega_1,\omega_2)\subseteq M^{p,q_2}_{(\omega)}(\rr {2n})$;

\vrum

\item the operator kernel $K$ of $a_t(x,D)$ belongs to
$M^p_{(\omega_0)}(\rr {2n})$ if
and only if $a\in M^p_{(\omega)}(\rr {2n})$ and for some constant $C$,
which only depends on $t$ and the involved weight functions, it holds
$\Vert K \Vert_{M^p_{(\omega_0)}}=C \Vert a \Vert_{M^p_{(\omega)}}$
\end{enumerate}
\end{prop}

\par

\begin{proof}
The assertion (1) is a restatement of Theorem 4.13 in \cite{Toft4}. The
assertion (2) follows by similar arguments as in the proof of
Proposition 4.8 in \cite{Toft4}, which we recall here. Let  $\chi ,
\psi\in \mathscr S(\rr {2n})$ be such that
$$
\psi(x,y) = \int _{\rr n}\chi ((1-t)x+ty,\xi)e^{i\scal {y-x}\xi}\, d\xi.
$$
By applying the Fourier inversion formula it follows by
straightforward computations that
$$
|\mathscr F(K \tau _{(x-ty, x+(1-t)y)}\psi )(\xi +(1-t)\eta ,
-\xi +t\eta )|=|\mathscr F(a\tau _{(x,\xi)}\chi )(y,\eta )|.
$$
The result now follows by applying the $L^p_{(\omega)}$ norm on these
expressions.
\end{proof}

\par

We also need the following proposition on continuity of linear
operators with kernels in modulation spaces.

\par

\begin{prop}\label{fourop3A}
Assume that  $p\in [1,\infty]$, $\omega_j\in \mathscr{P}(\rr {2n_j})$,
for $j=1,2$, and $\omega\in \mathscr{P}(\rr {2n_1+2n_2})$ fulfill for
some positive constant $C$
\begin{equation}\label{weight66}
\frac{\omega_2(x,\xi)}{\omega_1(y,-\eta)}\le C \omega(x,y,\xi,\eta). 
\end{equation}
Assume moreover that  $K\in M^p_{(\omega)}(\rr {n_1+n_2})$ and $T$ is
the linear and continuous  map from $\mathscr S(\rr {n_1})$ to
$\mathscr S'(\rr {n_2})$ defined by
\begin{equation}\label{opkercond}
(Tf)(x)= \scal{K(x,\cdot)}f
\end{equation}
when $f\in \mathscr S(\rr {n_1})$. Then $T$ extends uniquely to a
continuous  map from $M^{p'}_{(\omega_1)}(\rr{n_1})$ to
$M^p_{(\omega_2)}(\rr{n_2})$.

\par

On the other hand, assume that $T$ is a linear continuous map from
$M^{1}_{(\omega_1)}(\rr{n_1})$ to $M^\infty _{(\omega_2)}(\rr{n_2})$,
and that equality is attained in \eqref{weight66}.
Then there is a unique kernel $K\in M^\infty _{(\omega )}(\rr
{n_1+n_2})$ such that \eqref{opkercond} holds for every 
\end{prop}

\par

\begin{proof}
By Proposition \ref{p1.4} (3) and duality, it sufficies to prove that
for some constant $C$ independent of  $f\in \mathscr S(\rr {n_1})$ and
$g\in \mathscr S(\rr{n_2})$, it holds:
$$
|(K,g\otimes \bar f)|\le C\Vert K\Vert_{M^p_{(\omega)}}
\Vert g\Vert_{M^{p'}_{(1/\omega_2)}} \Vert f
\Vert_{M^{p'}_{(\omega_1)}}.
$$
Let $\omega_3(x,\xi)=\omega_1(x,-\xi)$. Then by straight-forward
calculation and using Remark \ref{p1.7} (7) we get
$$
\begin{array}{ll}
|(K,g\otimes \bar f)|&\le C_1\Vert K  \Vert_{M^{p}_{(\omega)}} \Vert
g\otimes \bar f \Vert_{M^{p'}_{(1/\omega)}} \le C_2 \Vert K
\Vert_{M^{p}_{(\omega)}} \Vert g  \Vert_{M^{p'}_{(1/\omega_2)}} \Vert
\bar f \Vert_{M^{p'}_{(\omega_3)}}
\\[1ex]
& \le C \Vert K  \Vert_{M^{p}_{(\omega)}}
\Vert g  \Vert_{M^{p'}_{(1/\omega_2)}} \Vert  f
\Vert_{M^{p'}_{(\omega_1)}}
\end{array}
$$

\par

The last part of the proposition concerning converse the property in
the case $p=\infty$ is a restatement of Proposition 4.7 in \cite
{Toft4} on generalization of Feichtinger-Gr{\"o}chenig's kernel
theorem.
\end{proof}

\par

\section{Continuity properties of Fourier
integral operators}\label{sec2}

\par

In this section we discuss Fourier integral operators with amplitudes
in modulation spaces, or more generally in certain types of coorbit
spaces. In Subsection \ref{ssec2.3} we extend Theorem 3.2 in
\cite{Bu0} to more general modulation spaces are involved.

\par

\subsection{Notation and general assumptions}\label{ssec2.1}
In the most general situation, we assume that the phase function
$\fy$ and the amplitude $a$ depend on $x\in \rr {n_1}$, $y\in \rr
{n_2}$ and $\zeta \in \rr m$, with dual variables $\xi \in \rr {n_1}$,
$\eta \in \rr {n_2}$ and $z \in \rr m$. For
conveniency we use the notation $N=n_1+n_2$, and $X, Y, Z,\dots$ for
tripples of the form $(x,y,\zeta )\in \rr {N+m}$, that is
\begin{equation}\label{trippl}
N=n_1+n_2,\qquad \text{and}\quad X=(x,y,\zeta )\in \rr {n_2}\oplus \rr
{n_1}\oplus \rr m \simeq \rr {N+m}.
\end{equation}
In order to state the results in Subsection \ref{ssec2.5} we also let
$F\in C^1(\rr {N+m})$, $V_1$ be a linear subspace of
$\rr {N+ m}$ of dimension $N$, $V_2=V_1^\bot$, and let $V_j'\simeq
V_j$ be the dual of $V_j$ for $j=1,2$. Also let any element
$X=(x,y,\zeta ) \in \rr {n_1+n_2+ m} = \rr {N+ m}$ and $(\xi ,\eta ,z )
\in  \rr {N+ m}$ in be written as
\begin{align*}
(x ,y ,\zeta ) &= \boldsymbol t _1e_1+\cdots +\boldsymbol t _{N}e_N +
\boldsymbol \varrho _{1}e_{N+1}\cdots +\boldsymbol \varrho _me_{N+m}
\\[1ex]
&\equiv  (\boldsymbol t,\boldsymbol \varrho ) = (\boldsymbol
t,\boldsymbol \varrho )_{V_1\oplus V_2}.
\intertext{and}
(\xi ,\eta ,z ) &= \boldsymbol \tau _1e_1+\cdots +\boldsymbol \tau
_{N}e_N + \boldsymbol u _{1}e_{N+1}\cdots +\boldsymbol u _me_{N+m}
\\[1ex]
&\equiv  (\boldsymbol \tau ,\boldsymbol u ) = (\boldsymbol \tau
,\boldsymbol u )_{V'_1\oplus V'_2}.
\end{align*}
for some orthonormal basis $e_1,\dots ,e_{N+m}$ in $\rr {N+m}$, and
let $F ' _{\boldsymbol \zeta}$ denotes the gradient of $F$ with
respect to the basis $e_{N+1},\dots ,e_{N+m}$.
Assume that $\mathsf d>0$, $\omega ,v\in \mathscr P(\rr {2(N+m)})$,
$\omega _j\in \mathscr P(\rr {2n_j})$ and $\fy \in C(\rr {N+m)}$ are
such that $v(X,\xi ,\eta ,z)=v(\xi ,\eta ,z)$ is submultiplicative,
$\fy '' \in M^{\infty ,1}_{(v)}(\rr {N+m})$ and that
\eqref{vvrel} and \eqref{weightsineq} hold.

\par

It is also convenient to let
\begin{align*}
\mathbf E_{a,\omega }(\boldsymbol t,\boldsymbol \varrho ,\boldsymbol
\tau ,\boldsymbol u) &= |V_\chi a(X,\xi ,\eta ,z)\omega (X,\xi ,\eta
,z)|,
\\[1ex]
{\mathcal E}_{a,\omega} (x,y,\boldsymbol u) &= \sup _{\zeta
,\boldsymbol \tau}\mathbf E_{a,\omega }(\boldsymbol t,\boldsymbol
\varrho ,\boldsymbol \tau ,\boldsymbol u),
\end{align*}
when $a\in \mathscr S'(\rr {N+m})$. Here $\boldsymbol u\in V_2'$ and
the supremum should be taken over $\zeta \in \rr m$ and $\boldsymbol
\tau \in V_1'$.

\par

Next we discuss general assumptions on the weight functions and phase
functions. In general we assume the phase function $\fy \in C(\rr
{N+m})$, and the weight functions $\omega ,v\in \mathscr P(\rr
{N+m}\times \rr {N+m})$, $\omega _0\in \mathscr P(\rr {2N})$ and
$\omega _1,\omega _2\in \mathscr P(\rr {2n_j})$,  are such that the
following conditions are fulfilled:
\begin{enumerate}
\item $v$ is submultiplicative and satisfies
\begin{equation}\label{vvrel}
\begin{aligned}
v(X,\xi ,\eta ,z)&=v(\xi ,\eta ,z),\qquad \text{and}
\\[1ex]
v(t\cdo ) &\le Cv \qquad X\in \rr {N+m},\ \xi \in \rr{n_2},\ \eta \in
\rr {n_1}\ z\in
\rr m,
\end{aligned}
\end{equation}
for some constant $C$ which is independent of $t\in [0,1]$.
In particular, $v(X,\xi ,\eta ,z)$ is constant with respect to $X\in
\rr {N+m}$;

\vrum

\item $\fy ^{(\alpha )}\in M^{\infty ,1}_{(v)}(\rr {N+m})$ for all
indices $\alpha$ such that $|\alpha |=2$ and
\begin{equation}\label{weightsineq}
\begin{aligned}
\frac {\omega _2(x,\xi )}{\omega _1(y,-\eta )} \le C_1\omega _0(x,y,\xi
,\eta ) &\le C_2\omega (X,\xi
-\fy '_x(X) , \eta -\fy '_y(X),-\fy '_\zeta (X)),
\\[1ex]
\omega (X,\xi _1+\xi _2,\eta _1+\eta _2,z_1+z_2) &\le C\omega (X,\xi
_1,\eta _1,z_1)v(\xi _2,\eta _2,z_2),
\\[1ex]
X = (x,y,\zeta )\in \rr {N+m},\quad \xi ,\, &\xi _1,\xi _2\in \rr
{n_2},\ \eta ,\,
\eta _1,\eta _2 \in \rr {n_1},\  z_1,z_2 \in \rr m,
\end{aligned}
\end{equation}
for some constants $C$, $C_1$ and $C_2$ which are independent of $X\in
\rr {N+m}$, $\xi _1,\xi _2\in \rr {n_1}$, $\eta _1,\eta _2\in \rr
{n_2}$ and $z_1,z_2\in \rr {m}$.
\end{enumerate}

\par

\subsection{The continuity assertions}\label{ssec2.2}
In most of our investigations we consider Fourier integral operator
$\op _\fy (a)$ where amplitude $a$ belongs to appropriate Banach space
which are defined in similar way as certain types of coorbit spaces in
Subsection \ref{ssec1.3}, and that the phase function $\fy$ should
satisfy the conditions in Subsection \ref{ssec2.1}. In this context we
find appropriate conditions on $a$, $\fy$ such that the conditions
(i)--(iii) below are fulfilled. Here the definition of admissible
pairs $(a,\fy )$ are presented in Subsection \ref{ssec2.5} below.
\begin{enumerate}
\item[{\rm{(i)}}] \emph{the pair $(a,\fy )$ is admissible, and the
kernel $K_{a,\fy}$ of $\op _\fy (a)$ belongs to $M^p_{(\omega _0)}$,
and
$$
\nm {K_{a,\fy}}{M^{p}_{(\omega _0)}} \le C{\mathsf d}^{-1}\exp (\nm
{\fy ''}{M^{\infty ,1}_{(v)}}) \nmm a ,
$$
for some constant $C$ which is independent of $a\in \mathscr S'(\rr
{N+m})$ and $\fy \in C(\rr {N+m})$;}

\vrum

\item[{\rm{(ii)}}] \emph{the definition of $\op _\fy (a)$ extends
uniquely to a continuous operator from $M^{p'}_{(\omega _1)}(\rr
{n_1})$ to $M^{p}_{(\omega _2)}(\rr {n_2})$. Furthermore, for some
constant $C$ it holds
$$
\nm {\op _\fy (a)}{M^{p'}_{(\omega _1)}\to M^{p}_{(\omega _2)}} \le
C{\mathsf d}^{-1}\exp (\nm {\fy ''}{M^{\infty ,1}_{(v)}}) \nmm a \,
\text ;
$$
}

\vrum

\item[{\rm{(iii)}}] \emph{if in addition $1\le p\le 2$, then $\op _\fy
(a)\in \mathscr I_p(M^2_{(\omega _1)},M^2_{(\omega _2)})$.}
\end{enumerate}

\par

\subsection{Reformulation of Fourier integral operators in terms of
short time Fourier transforms}\label{ssec2.3}
For each real-valued $\fy \in C(\rr {N+m})$ which satisfies $\fy
^{(\alpha )}\in M^{\infty ,1}_{(v)}(\rr {N+m})$ for all multi-indices
$\alpha$ such that $|\alpha |=2$, and $a\in \mathscr S(\rr {N+m})$, it
follows that the Fourier integral operator $f\mapsto \op _\fy (a)f$ in
\eqref{fintop} is well-defined and makes sense as a continuous
operator from $\mathscr S(\rr {n_1})$ to $\mathscr S'(\rr {n_2})$,
that is
$$
(\op _\fy (a)f,g)=(2\pi )^{-N/2}\int _{\rr {N+m}}a(X)e^{i\fy
(X)}f(y)\overline{g(x)}\, dX,
$$
when $f\in \mathscr S(\rr {n_1})$ and $g\in \mathscr S(\rr {n_2})$.
In order to extend the definition we reformulate the latter relation
in terms of short-time Fourier transforms.

\par

Assume that $0\le \chi ,\psi \in C_0^\infty (\rr {N+m})$ and $0\le
\chi _j\in C_0^\infty (\rr {n_j})$ for $j=1,2$ satisfy
$$
\nm {\chi _j} {L^1}=\nm {\chi }{L^2}=1,
$$
and let $X_1=(x_1,y_1,\zeta _1)\in \rr {N+m}$. By straight-forward
computations we get
\begin{multline*}
(2\pi )^{N/2}(\op _\fy (a)f,g) = \int _{\rr {N+m}}a(X)f(y)\overline
{g(x)} e^{i\fy (X )}\, dX
\\[1ex]
= \iint _{\rr {2(N+m)}} a(X+X_1)\chi (X_1)^2f(y+y_1)\chi
_1(y_1)\overline {g(x+x_1)\chi _2(x_1)} e^{i\psi (X_1)\fy (X+X_1 )}\,
dXdX_1
\end{multline*}
Then Parseval's formula gives
\begin{multline*}
(2\pi )^{N/2}(\opp _\fy (a)f,g)
\\[1ex]
=\iiiint _{\rr {2(N+m)}} F(X,\xi ,\eta ,\zeta _1) \mathscr F(f(y+\cdo
)\chi _1)(-\eta )\overline {\mathscr F(g(x+\cdo )\chi _2)(\xi )}\,
dXd\xi d\eta d\zeta _1
\\[1ex]
=\iiiint _{\rr {2(N+m)}} F(X,\xi ,\eta ,\zeta _1) (V_{\chi
_1}f)(y,-\eta )\overline {(V_{\chi _2}g)(x,\xi )}e^{-i(\scal
x\xi+\scal y{\eta})}\, dXd\xi d\eta d\zeta _1,
\\[1ex]
=\iiint _{\rr {2N+m}}\Big ( \int _{\rr m} F(X,\xi ,\eta ,\zeta _1)\,
d\zeta _1 \Big )(V_{\chi _1}f)(y,-\eta )\overline {(V_{\chi
_2}g)(x,\xi )}e^{-i(\scal x\xi+\scal y{\eta})}\, dXd\xi d\eta ,
\end{multline*}
where
\begin{equation*}
F(X,\xi ,\eta ,\zeta _1) = \mathscr F_{1,2}\big (e^{i\psi (\cdo ,\zeta
_1)\fy (X+(\cdo ,\zeta _1 ))}a(X+(\cdo ,\zeta _1))\chi (\cdo ,\zeta
_1)^2\big )(\xi ,\eta ).
\end{equation*}
Here $\mathscr F _{1,2}a$ denotes the partial Fourier transform
of $a(x,y,\zeta )$ with respect to the $x$ and $y$ variables.

\par

By Taylor's formula it follows that
$$
\psi (X_1) \fy (X+X_1) =\psi (X_1)\psi _{1,X}(X_1) +\psi _{2,X}(X_1),
$$
where
\begin{equation}\label{psidef}
\begin{aligned}
\psi _{1, X}( X_1) &= \fy ( X)+\scal {\fy '( X)}{ X_1}\quad \text{and}
\\[1ex]
\psi _{2, X}( X_1) &= \psi ( X_1)\int _0^1
(1-t)\langle \varphi ''( X+t X_1)
X_1, X_1\rangle \, dt .
\end{aligned}
\end{equation}
By inserting these expressions into the definition of $F(X,\xi ,\eta
,\zeta _1)$, and integrating with respect to the $\zeta _1$-variable
give
\begin{multline*}
\int _{\rr m}F(X,\xi ,\eta ,\zeta _1)\, d\zeta _1
\\[1ex]
= (2\pi )^{m/2}\mathscr F((e^{i\psi _{2,X}}\chi ) (a(\cdo +X)\chi
)(\xi -\fy '_x(X),\eta -\fy '_y(X),-\fy '_\zeta (X))
\\[1ex]
= (2\pi )^{N/2}\mathcal H_{a,\fy}(X,\xi ,\eta ),
\end{multline*}
where
\begin{equation}\label{Hdef}
\begin{aligned}
\mathcal H_{a,\fy}(X,\xi ,\eta )
= h_X*(\mathscr F(a(\cdo +&X)\chi ))(\xi -\fy '_x(X),\eta -\fy
'_y(X),-\fy '_\zeta (X)),
\\[1ex]
\text{and}\quad h_X &= (2\pi )^{-(N-m/2)} (\mathscr F(e^{i\psi
_{2,X}}\chi ))
\end{aligned}
\end{equation}

\par

Summing up we have proved that
\begin{equation}\label{fourrel}
\begin{aligned}
(&\operatorname{Op} _{\fy} (a)f,g) =T_{a,\fy}(f,g)
\\[1ex]
&\equiv \iiint  \mathcal
H_{a,\fy}(X,\xi ,\eta )
(V_{\chi _0}f)(y,-\eta )\overline {(V_{\chi _0}g)(x,\xi )}e^{-i(\scal
x\xi+\scal y{\eta})}\, dXd\xi d\eta .
\end{aligned}
\end{equation}

\par

\subsection{An extension of a result by Boulkhemair}\label{ssec2.4}
Next we consider Fourier integral operators with amplitudes in the
modulation space $M^{\infty ,1}_{(\omega )}(\rr {2n+m})$, where we are
able to state and prove the announced generalization of Theorem 3.2 in
\cite {Bu1}. Here we assume that $n_1=n_2=n$ which implies that
$N=2n$.

\par

\begin{thm}\label{boulkemA}
Assume that $1<p<\infty$, and that $\fy \in C(\rr
{2n+m})$, $\omega ,v\in \mathscr P(\rr {2(N+m)})$ and $\omega
_1,\omega _2\in \mathscr P(\rr {2n})$ fulfill the conditions in
Subsection \ref{ssec2.1}. Also assume that \eqref{detphicond} holds for some $\mathsf d >0$. Then the following is true:
\begin{enumerate}
\item the map $a\mapsto \op _\fy (a)$ from $\mathscr S(\rr {2n+m})$ to
$\mathscr L(\mathscr S(\rr n),\mathscr S'(\rr n))$ extends uniquely to
a continuous map from $M^{\infty ,1}_{(\omega )}(\rr {2n+m})$ to
$\mathscr L(\mathscr S(\rr n),\mathscr S'(\rr n))$;

\vrum

\item if $a\in M^{\infty ,1}_{(\omega )}(\rr {2n+m})$, then the map
$\op _{\fy }(a)$ from $\mathscr S(\rr n)$ to $\mathscr S'(\rr n)$ is
uniquely extendable to a continuous operator from $M^p_{(\omega
_1)}(\rr n)$ to $M^p_{(\omega _2)}(\rr n)$. Moreover, for some
constant $C$ it holds
\begin{equation}\label{normuppsk2}
\nm {\op _{\fy }(a)}{M^p_{(\omega _1)}\to M^p_{(\omega _2)}}\le
C\mathsf d ^{-1}\nm a{M^{\infty ,1}_{(\omega )}}\exp (C\nm {\fy
''}{M^{\infty 1,}_{(v)}}).
\end{equation}
\end{enumerate}
\end{thm}

\par

The proof needs some preparing lemmas.

\par

\begin{lemma}\label{lemmafourop3}
Assume that $v(x,\xi )=v(\xi )\in \mathscr P(\rr {n})$ is
submultiplicative and satisfies $v(t\xi )\le Cv(\xi )$ for some
constant $C$ which is independent of $t\in [0,1]$ and $\xi \in \rr
n$. Also assume that $f\in M^{\infty ,1}_{(v)}(\rr n)$, $\chi \in
C^\infty _0(\rr n)$ and that $x\in \rr n$, and let
$$
\fy _{x,j,k}(y)=\chi (y)\int _0^1 (1-t)f(x+ty)y_jy_k\, dt .
$$
Then there is a constant $C$ and a function $g\in M^1_{(v)}(\rr n)$
such that $\nm {g}{M^1_{(v)}}\le C\nm f{M^{\infty ,1}_{(v)}}$ and
$|\mathscr F(\fy _{x,j,k})(\xi )|\le \widehat g(\xi )$.
\end{lemma}

\par

\begin{proof}
We first prove the assertion when $\chi$ is replaced by $\psi
_0(y)=e^{-2|y|^2}$. For conveniency we let 
$$
\psi _1(y)= e^{-|y|^2},\quad \text{and}\quad \psi _2(y) = e^{|y|^2},
$$
and 
$$
H_{\infty ,f} (\xi )\equiv \sup _{x}|\mathscr F(f\, \psi _1(\cdo
-x))(\xi )|.
$$
We claim that $g$, defined by
\begin{equation}\label{defg}
\widehat g(\xi)=\int_0^1\int _{\rr n}(1-t)H_{\infty,f}(\eta)e^{-|\xi-
t\eta|^2/16}\,d\eta dt,
\end{equation}
fulfills the required properties.

\par

In fact, by applying $M^1_{(v)}$ norm on $g$ and using Minkowski's
inequality and Remark \ref{p1.7} (6), we obtain
\begin{align*}
\nm g{M^1_{(v)}} &= \Big \Vert \int_0^1\int _{\rr n}
(1-t)H_{\infty,f}(\eta)e^{-|\xi- t\eta|^2/16}\,d\eta dt \Big \Vert
_{M^1_{(v)}}
\\[1ex]
&\le \int_0^1\int _{\rr n} (1-t)H_{\infty,f}(\eta)\nm {e^{-|\cdo -
t\eta|^2/16}}{M^1_{(v)}}\,d\eta dt
\\[1ex]
&\le C_1\int_0^1\int _{\rr n} (1-t)H_{\infty,f}(\eta)\nm {e^{-|\cdo
|^2/16}}{M^1_{(v)}}v(t\eta )\,d\eta dt
\\[1ex]
&\le C_2\int_0^1\int _{\rr n} (1-t)H_{\infty,f}(\eta)v(\eta )\nm
{e^{-|\cdo |^2/16}}{M^1_{(v)}}\,d\eta dt
\\[1ex]
&= C_3\nm {H_{\infty,f}v}{L^1} = C_3\nm f{M^{\infty ,1}_{(v)}}.
\end{align*}

\par

In order to prove that $|\mathscr F(\fy _{x,j,k})(\xi )|\le g(\xi )$,
we let $\psi (y)=\psi _{j,k}(y)= y_jy_k\psi _0 (y)$. Then
$$
\fy _{x,j,k}(y)= \psi (y)\int _0^1 (1-t)f(x+ty)\, dt.
$$
By a change of variables we obtain
\begin{equation}\label{estfop3}
\begin{aligned}
|\mathscr F(\fy _{x,j,k})(\xi )| &= \Big |\int _0^1 (1-t)\Big ( \int
_{\rr n} f(x+ty)\psi (y)e^{-i\scal y\xi}\, dy\Big )\, dt \Big |
\\[1ex]
&= \Big |\int _0^1 t^{-n}(1-t)\mathscr F(f\, \psi ((\cdo -x)/t))(\xi
/t)e^{i\scal x\xi /t}\, dt\Big |
\\[1ex]
&\le \int _0^1 t^{-n}(1-t)\sup _{x\in \rr n}|\mathscr F(f\,
\psi ((\cdo -x)/t))(\xi /t)|\, dt.
\end{aligned}
\end{equation}
We need to estimate the right-hand side. By straight-forward
computations we get
\begin{multline*}
|\mathscr F(f\, \psi ((\cdo -x)/t))(\xi )|
\\[1ex]
\le  (2\pi )^{-n/2}\big
( |\mathscr F(f\, \psi _1(\cdo -x))|*| \mathscr F(\psi
((\cdo -x)/t)\, \psi _2(\cdo -x))|\big )(\xi )
\\[1ex]
=(2\pi )^{-n/2}\big
( |\mathscr F(f\, \psi _1(\cdo -x))|*| \mathscr F(\psi
(\cdo /t)\, \psi _2)|\big )(\xi )
\end{multline*}
where the convolutions should be taken with respect to the
$\xi$-variable only. Then
\begin{equation}\label{estfop33}
|\mathscr F(f\, \psi ((\cdo -x)/t))(\xi )|
\le (2\pi )^{-n/2}\big
( H_{\infty ,f}*| \mathscr F(\psi (\cdo /t)\, \psi _2)|\big )(\xi )
\end{equation}
For the estimate of the latter Fourier transform we observe that
\begin{equation}\label{gaussest1}
| \mathscr F(\psi (\cdo /t)\, \psi _2)| =  |\partial _j\partial _k
\mathscr F(\psi _0(\cdo /t)\, \psi _2)|.
\end{equation}
Since $\psi _0$ and $\psi _2$ are Gauss functions and $0\le t\le 1$, a
straight-forward computation gives
\begin{equation}\label{gaussest2}
\mathscr F(\psi _0(\cdo /t)\, \psi _2)(\xi ) = \pi
^{n/2}t^n(2-t^2)^{-n/2}e^{-t^2|\xi |^2/(4(2-t^2))}.
\end{equation}
Thus a combination of \eqref{gaussest1} and \eqref{gaussest2}
therefore give
\begin{equation}\label{gaussest3}
| \mathscr F(\psi (\cdo /t)\, \psi _2)(\xi )|\le
Ct^ne^{-t^2|\xi |^2/16},
\end{equation}
for some constant $C$ which is independent of $t\in [0,1]$. The
assertion now follows by combining \eqref{estfop3}, \eqref{estfop33}
and \eqref{gaussest3}.

\par

In order  to prove the result for general $\chi\in C^\infty_0(\rr n)$
we set
$$
h_{x,j,h}(y)=\psi_0(y)\int_0^1 (1-t)f(x+ty)y_jy_k\, dt,
$$
and we observe that the result is already proved  when
$\varphi_{x,j,k}$ is replaced by $h_{x,j,h}$ and moreover $\varphi
_{x,j,k}=\chi_1 h_{x,j,k}$, for some $\chi_1\in C^\infty_0(\rr n)$.
Hence if $g_0$ is given as the right-hand side  of \eqref{defg}, the
first part of the proof  shows that
$$
\left| \mathscr F(\varphi_{x,j,k})(\xi)\right|=\left| \mathscr
F(\chi_1 h_{x,j,k}(\xi))\right|\le (2\pi)^{-n/2}|\widehat{\chi_1}|\ast
g-0(\xi)\equiv g(\xi).
$$
Moreover $\Vert g_0 \Vert_{M^1_{(v)}} \le C\Vert f \Vert_{M^{\infty,
1}_{(v)}}$.

\par

Since $M^1_{(v)}\ast L^1_{(v)}\subseteq M^1_{(v)}$, we get for some
positive constants $C, C_1$
$$
\Vert g \Vert_{M^1_{(v)}}\le C\Vert
\widehat{\chi_1}\Vert_{L^1_{(v)}}\Vert g_0 \Vert_{M^1_{(v)}}\le C_1
\Vert f \Vert_{M^{\infty, 1}_{(v)}},
$$
which proves the result
\end{proof}

\par

As a consequence of Lemma \ref{lemmafourop3} we have the following
result.

\par

\begin{lemma}\label{lemmafourop33}
Assume that $v(x,\xi )=v(\xi )\in \mathscr P(\rr {n})$ is
submultiplicative and satisfies $v(t\xi )\le Cv(\xi )$ for some
constant $C$ which is independent of $t\in [0,1]$ and $\xi \in \rr
n$. Also assume that $f_{j,k}\in M^{\infty ,1}_{(v)}(\rr n)$ for
$j,k=1,\dots, n$, $\chi \in
C^\infty _0(\rr n)$ and that $x\in \rr n$, and let
$$
\fy _{x}(y)=\sum _{j,k=1,\dots ,n}\fy _{x,j,k}(y),\quad
\text{where}\quad \fy _{x,j,k}(y)=\chi (y)\int _0^1
(1-t)f_{j,k}(x+ty)y_jy_k\, dt .
$$
Then there is a constant $C$ and a function $\Psi \in M^1_{(v)}(\rr
n)$ such that
$$
\nm {\Psi}{M^1_{(v)}}\le \exp(C\sup _{j,k}\nm {f_{j,k}}{M^{\infty
,1}_{(v)}})
$$
and
\begin{equation}\label{psiestim}
|\mathscr F(\exp (i\fy _{x})(\xi ))|\le (2\pi )^{n/2}\delta _0
+\widehat \Psi (\xi ).
\end{equation}
\end{lemma}

\par

\begin{proof}
By Lemma \ref{lemmafourop3}, we may find a function $g\in M^1_{(v)}$
and a constant $C>0$ such that
$$
|\widehat {\fy _x}(\xi )| \le \widehat g(\xi ),\qquad \nm
g{M^1_{(v)}}\le C \sup _{j,k}\left(\nm {f_{j,k}}{M^{\infty
,1}_{(v)}}\right).
$$
Set
\begin{equation*}
\begin{alignedat}{3}
\Phi _{0,x} &= (2\pi )^{n/2}\delta _0,&\qquad \Phi _{l,x} &= |\mathscr
F(\fy _{x})| * \cdots *|\mathscr F(\fy _{x})|,&\qquad l&\ge 1
\\[1ex]
\Upsilon _0 &= (2\pi )^{n/2}\delta _0,&\qquad \widehat \Upsilon _l &=
\widehat g*\cdots *\widehat g,&\qquad l&\ge 1,
\end{alignedat}
\end{equation*}
with $l$ factors in the convolutions. Then by Taylor expansion,
there is a positive constant $C$ such that
\begin{equation*}
|\mathscr F(\exp (i\fy _{x}(\cdo ))(\xi ))| \le \sum _{l=0}^\infty
C^{l}\Phi _{l,x}/l! \le  \sum _{l=0}^\infty C^{l}\widehat \Upsilon
_l/l!.
\end{equation*}
Hence, if
$$
\Psi \equiv \sum _{l=1}^\infty C^{l}\Upsilon _l/l! ,
$$
then \eqref{psiestim} follows. Furthermore, since $v$ is
submultiplicative, by means of Proposition \ref{multprop} we obtain
$$
\nm {\Upsilon _l}{M^{1}_{(v)}} = (2\pi )^{(l-1)n/2}\nm {g \cdots
g}{M^{1}_{(v)}} \le (C_1\nm g{M^{1}_{(v)}})^l,\qquad l\ge 1,
$$
for some positive constant $C_1$. This gives
\begin{multline*}
\nm \Psi{M^1_{(v)}}\le \sum _{l=1}^\infty \nm {\Upsilon
_l}{M^{1}_{(v)}}/l! \le \sum _{l=1}^\infty (C_1\nm
g{M^{1}_{(v)}})^l/l!
\\[1ex]
\le \sum _{l=1}^\infty (C_2\sup _{j,k}(\nm {f_{j,k}}{M^{\infty
,1}_{(v)}})^l/l!
\le \exp (C_2\sup _{j,k}\nm {f_{j,k}}{M^{\infty ,1}_{(v)}}),
\end{multline*}
for some constants $C_1$ and $C_2$, and the assertion is proved.
\end{proof}

\par

\begin{proof}[Proof of Theorem \ref{boulkemA}]
We shall mainly follow the proof of Theorem 3.2 in \cite{Bu1}. First
assume that $a\in C_0^\infty (\rr {2n+m})$ and $f,g\in \mathscr S(\rr
n)$. Then it follows that $\op _\fy (a)$ makes sense as a continuous
operator from $\mathscr S$ to $\mathscr S'$. Since
$$
|\mathscr F(e^{i\Psi _{2,X}}\chi )|\le (2\pi )^{-n+m/2}|\mathscr
F(e^{i\Psi _{2,X}})|*|\widehat \chi |,\qquad |\mathscr F(a(\cdo
+X)\chi )| = |V_\chi a(X,\cdo )|,
$$
if follows from Lemma \ref{lemmafourop33} that
\begin{equation}\label{Haest}
|\mathcal H_{a,\fy}(X,\xi ,\eta )|\le C (G*|V_\chi a(X,\cdo )|)(\xi
-\fy '_x(X), \eta -\fy '_y(X),-\fy '_\zeta (X)),
\end{equation}
for some non-negative $G\in L^1_{(v)}$ which satisfies $\nm
G{L^1_{(v)}}\le C\exp (C\nm {\fy ''}{M^{\infty ,1}_{(v)}}$. Here
$\mathcal H_{a,\fy}$ is the same as in  \eqref{Hdef}.

\par

Next we set
\begin{equation}\label{Efuncdef}
\begin{aligned}
\mathbf E_{a,\omega }(X,\xi ,\eta ,z) &= |V_\chi a(X,\xi ,\eta
,z)\omega (X,\xi ,\eta ,z)|,\qquad
\\[1ex]
\widetilde {\operatorname E}_{a,\omega}(\xi ,\eta ,z) &= \sup
_{X}\mathbf E_{a,\omega }(X,\xi ,\eta ,z)
\\[1ex]
F_1(y,\eta ) &= |V_{\chi _1}f(y,\eta )\omega _1(y,\eta )|,
\\[1ex]
F_2(x,\xi ) &= |V_{\chi _2}g(x,\xi )/\omega _2(x,\xi )|,
\end{aligned}
\end{equation}
and
\begin{align}
\mathcal Q_{a,\omega}(\mathsf X,\zeta ) &= \mathbf E_{a,\omega }(X,\xi
-\fy '_x(X),\eta -\fy '_y(X),-\fy '_\zeta (X)),\label{Qadef}
\intertext{and}
\mathcal R_{a,\omega ,\fy}(\mathsf X,\zeta ) &= ((Gv)*\mathbf
E_{a,\omega }(X,\cdo ))(\xi -\fy '_x(X),\eta -\fy '_y(X),-\fy '_\zeta
(X)),\label{Radef}
\\[1ex]
X &= (x,y,\zeta ),\quad \mathsf X = (x,y,\xi ,\eta ).\notag
\end{align}
Note here the difference between $X$ and $\mathsf X$. By combining
\eqref{weightsineq} with \eqref{Haest} we get
\begin{multline*}
\iiint | \mathcal H_{a,\fy}(X,\xi ,\eta ) (V_{\chi _1}f)(y,-\eta
){(V_{\chi _2}g)(x,\xi )} |\, dXd\xi d\eta
\\[1ex]
\le
C_1\iiint (G*|V_\chi a(X,\cdo )|)(\xi -\fy '_x(X), \eta -\fy
'_y(X),-\fy '_\zeta (X))\times
\\[1ex]
|(V_{\chi _1}f)(y,-\eta ){(V_{\chi _2}g)(x,\xi )}|\, dXd\xi d\eta
\\[1ex]
\le
C_2\iiint \mathcal R_{a,\omega ,\fy}(\mathsf X,\zeta )
F_1(y,-\eta )F_2(x,\xi )\, dXd\xi d\eta
\end{multline*}
Summing up we have proved that
\begin{equation}\label{estfourop1}
|(\op _\fy (a)f,g)| \le C \iiint \mathcal R_{a,\omega ,\fy}(\mathsf
X,\zeta ) F_1(y,-\eta )F_2(x,\xi )\, dXd\xi d\eta .
\end{equation}

\par

It follows from \eqref{Efuncdef}, \eqref{Radef}, \eqref{estfourop1}$'$
and H{\"o}lder's inequality that
\begin{equation}\label{eq.11}
|(\op _\fy (a)f,g)|\le C J_1\cdot J_2,
\end{equation}
where
\begin{align*}
J_1 &= \Big ( \iiint (Gv)\ast \widetilde {\operatorname E} _{a,\omega}
(\xi -\fy '_x(X) ,\eta -\fy '_y(X),-\fy '_\zeta (X)) F_1(y,-\eta )^p\,
dXd\xi d\eta \Big )^{1/p}
\\[1ex]
J_2 &= \Big ( \iiint ((Gv)\ast \widetilde {\operatorname E}
_{a,\omega} (\xi -\fy '_x(X) , \eta -\fy '_y(X),-\fy '_\zeta (X))
F_2(x,\xi )^{p'}\, dXd\xi d\eta \Big )^{1/p'}.
\end{align*}
We have to estimate $J_1$ and $J_2$. By taking $z =\fy '_\zeta (X)$,
$\zeta _0=\fy '_y(X)$, $y$, $\xi$ and $\eta$ as new variables of
integrations, and using \eqref{detphicond}, it follows that
\begin{multline*}
J_1 \le  \Big (  \mathsf d ^{-1}\iiint (Gv)\ast \widetilde
{\operatorname E}_{a,\omega} (\xi -\kappa _1(y,z,z_0)
, \eta -\zeta _0,z) F_1(y,-\eta )^p\, dy dz d\xi d\eta d\zeta _0\Big
)^{1/p}
\\[1ex]
= \Big (  \mathsf d ^{-1}\iiint (Gv)\ast \widetilde {\operatorname E}
_{a,\omega} (\xi, \zeta _0,z) F_1(y,-\eta )^p\, dy dz d\xi d\eta
d\zeta _0\Big )^{1/p}
\\[1ex]
= \mathsf d ^{-1/p}\nm {(Gv)\ast \widetilde {\operatorname E}
_{a,\omega }}{L^1}^{1/p}\nm {F_1}{L^p},
\end{multline*}
for some continuous function $\kappa _1$. It follows from Young's
inequality and \eqref{weightsineq} that
\begin{equation*}
\Vert (Gv)\ast \widetilde {\operatorname E} _{a,\omega}\Vert_{L^1}\le
\Vert G\Vert_{L^1_{(v)}}\Vert \widetilde {\operatorname E} _{a,
\omega} \Vert_{L^1}.
\end{equation*}
Hence
\begin{equation}\label{eq.22}
J_1\le \mathsf d ^{-1/p}\left(C\exp
(C\nm {\fy ''}{M^{\infty ,1}_{(v)}})\nm {a}{M^{\infty ,1}_{(\omega
)}}\right)^{1/p}\nm {f}{M^p_{(\omega _1)}}.
\end{equation}
If we instead take $x$, $y_0=\fy '_3(X)$, $\xi $, $\eta $ and $\zeta
_0=\fy '_1(X)$ as new variables of integrations, it follows by similar
arguments that
\begin{equation}\tag*{(\ref{eq.22})$'$}
J_2\le \mathsf d ^{-1/p'}\left (C\exp (C\nm {\fy ''}{M^{\infty
,1}_{(v)}})\nm {a}{M^{\infty ,1}_{(\omega )}}\right)^{1/p'}\nm
{g}{M^{p'}_{(1/\omega _2)}}.
\end{equation}
A combination of \eqref{eq.11}, \eqref{eq.22} and \eqref{eq.22}$'$ now
gives
$$
|(\op _{\fy}(a)f,g)| \le  C\mathsf d ^{-1}\nm a{M^{\infty
,1}_{(\omega )}} \nm {f}{M^p_{(\omega _1)}} \nm {g}{M^{p'}_{(1/\omega
_2)}}\exp (C\nm {\fy ''}{M^{\infty ,1}_{(v)}}),
$$
which proves \eqref{normuppsk2}, and the result follows when $a\in
C^\infty_0(\rr{2n+m})$ and $f,g\in\mathscr S(\rr n)$.

\par

Since $\mathscr S$ is dense in $M^p_{(\omega _1)}$ and
$M^{p'}_{(1/\omega _2)}$, the result also holds for $a\in C_0^\infty $
and $f\in M^p_{(\omega _1)}$. Hence it follows by Hahn-Banach's
theorem that the asserted extension of the map $a\mapsto \op _\fy (a)$
exists.

\par

It remains to prove that this extension is unique. Therefore assume
that $a\in M^{\infty ,1}_{(\omega )}$ is arbitrary, and take a
sequence $a_j\in C_0^\infty$ for $j=1,2,\dots$ which converges
to $a$ with respect to the narrow convergence. Then $\widetilde
{\operatorname  E}_{a_j,\omega}$ converges to $\widetilde
{\operatorname E} _{a,\omega}$ in $L^1$ as
$j$ turns to infinity. By \eqref{psidef}--\eqref{fourrel} and the
arguments at the above, it follows from Lebesgue's theorem that
$$
(\op _\fy (a_j)f,g)\to (\op _\fy (a)f,g)
$$
as $j$ turns to infinity. This proves the uniqueness, and the result
follows.
\end{proof}

\par

\subsection{Fourier integral operators with amplitudes in coorbit
spaces}\label{ssec2.5}
A crucial point concerning the uniqueness when extending the
definition of $\op _\fy$ to amplitudes in $M^{\infty ,1}_{(\omega )}$
in Theorem \ref{boulkemA} is that $C_0^\infty$ is dense in $M^{\infty
,1}_{(\omega )}$ with respect to the narrow convergence. On the other
hand, the uniqueness of the extension of the definition might be
violated when spaces of amplitudes are considered where such density
or duality properties are missing. In the present paper we use the
reformulation \eqref{fourrel} to extend the definition of the Fourier
integral operator in \eqref{fintop} to certain amplitues which are not
contained in $M^{\infty ,1}_{(\omega )}$.

\par

More precisely, assume that $a\in \mathscr S'(\rr {N+m})$, $f\in \rr
{n_1}$, $g\in \mathscr S(\rr {n_2})$ and that the mapping
$$
(X,\xi ,\eta )\mapsto \mathcal H_{a,\fy}(X,\xi ,\eta ) (V_{\chi
_0}f)(y,-\eta )\overline {(V_{\chi _0}g)(x,\xi )}
$$
belongs to $L^1(\rr {N+m}\times \rr {N})$. (Here recall that
$N=n_1+n_2$, where, from now on, $n_1$ and $n_2$ might be different.)
Then we let $T_{a,\fy}(f,g)$ be defined as the right-hand side of
\eqref{fourrel}.

\par

\begin{defn}\label{deffourop}
Assume that $N=n_1+n_2$, $v\in \mathscr P(\rr {N+m}\times \rr {N+m})$
is submultiplicative and satisfies \eqref{vvrel}, $\fy \in C^2(\rr
{N+m})$ is such and that $\fy ^{(\alpha )}\in M^{\infty
,1}_{(v)}$ for all multi-indices $\alpha$ such that $|\alpha |=2$, and
that $a\in \mathscr S'(\rr {N+m})$ is such that
$f\mapsto T_{a,\fy}(f,g_0)$ and $g\mapsto T_{a,\fy}(f_0,g)$ are
well-defined and continuous from $\mathscr S(\rr {n_1})$ and from
$\mathscr S(\rr {n_2})$ respectively to $\mathbf C$, for each fixed
$f_0\in \mathscr S(\rr {n_1})$ and $g_0\in \mathscr S(\rr
{n_2})$. Then the pair $(a,\fy )$ is called \emph{admissible}, and the
Fourier integral operator $\op _\fy (a)$ is the linear continuous
mapping from $\mathscr S(\rr {n_1})$ to $\mathscr S'(\rr {n_2})$ which
is defined by the formulas \eqref{psidef}, \eqref{Hdef} and
\eqref{fourrel}.
\end{defn}

\par

Here recall that if for each fixed $f_0\in \mathscr S(\rr {n_1})$ and
$g_0\in \mathscr S(\rr {n_2})$, the mappings $f\mapsto T(f,g_0)$ and
$g\mapsto T(f_0,g)$ are continuous from $\mathscr S(\rr {n_1})$ and
from $\mathscr S(\rr {n_2})$ respectively to $\mathbf C$, then it
follows by Banach-Steinhauss theorem that $(f,g)\mapsto T(f,g)$ is
continuous from $\mathscr S(\rr {n_1})\times \mathscr S(\rr {n_2})$ to
$\mathbf C$.

\par

The following theorem involves Fourier integral operators with
amplitudes which are not contained in $M^{\infty ,1}_{(\omega )}$.

\par

\begin{thm}\label{thm3.1A}
Assume that $N$, $\chi$, $\omega$, $\omega _j$, $v$, $\fy$, $V_j$,
$V_j'$, $\boldsymbol \varrho$ $\boldsymbol \tau$ and $\boldsymbol u$
for $j=0,1,2$ are the same as in Subsection \ref{ssec2.1}. Also assume
that $a\in \mathscr S'(\rr {N+m})$ fulfills $\nmm a<\infty$, where
$$
\nmm a \equiv \essup{x,y} \Big (\int _{V_2'}\big (\sup _{\zeta \in \rr
m,\, \boldsymbol \tau \in V_1'}|V_\chi a(X,\xi ,\eta ,z)|\big )\,
d\boldsymbol u\Big ),
$$
and that $|\det (\fy ''_{\boldsymbol \varrho ,\zeta})|\ge \mathsf d$
for some $\mathsf d>0$. Then {\rm{(i)--(ii)}} in Subsection
\ref{ssec2.2} hold for $p=\infty$.
\end{thm}

\par

We note that the conditions on $a$ in Theorem \ref {thm3.1A} means
that $a$ should belong to a subspace of $M^\infty _{(\omega )}$ which
is a superspace of $M^{\infty ,1}_{(\omega )}$.
Roughly speaking it follows that $a$ should belong to a weighted space
$M^\infty$  in some variables and to a superspace of $M^{\infty
,1}$ in the other variables.

\par

\begin{proof}
It suffices to prove that (i) in Subsection \ref{ssec2.2}.

\par

The notations are similar to that in the proof of Theorem
\ref{boulkemA}. Furthermore we let
\begin{equation*}
{\mathcal E}_{a,\omega} (x,y,\boldsymbol u) = \sup _{\zeta
,\boldsymbol \tau}\mathbf E_{a,\omega }(X,\xi ,\eta ,z),\quad
\text{and}\quad G_{1,v}(\boldsymbol u) = \int _{V_1'}G(\xi ,\eta ,z)\,
d\boldsymbol \tau,
\end{equation*}
where $\mathbf E$ is given by \eqref{Efuncdef}. By taking $x,y,-\fy
'_{\boldsymbol \varrho} (X),\xi ,\eta $ as new variables of
integration in \eqref{estfourop1}, and using the fact that $|\det (\fy
''_{\boldsymbol \varrho ,\zeta})|\ge \mathsf d$ we get
\begin{multline}\label{estagain4}
|(\op _\fy (a)f,g)| \le C{\mathsf d}^{-1}\int _{\rr {2N}}\mathcal
K_{a,\omega ,Cv} (\mathsf X) F_1(y,-\eta )F_2(x,\xi )\, d\mathsf X
\\[1ex]
\le C{\mathsf d}^{-1}\nm {\mathcal K_{a,\omega ,Gv}}{L^\infty}\nm
{F_1}{L^1}\nm {F_2}{L^1},
\end{multline}
where $\mathsf X=(x,y,\xi ,\eta )$ and
$$
\mathcal K_{a,\omega ,Gv}(\mathsf X) = \int _{V_2'} \big
((Gv)*(\mathbf E_{a,\omega}(x,y,\kappa _1,\cdo
))\big )((\xi ,\eta ,0)_{\rr {N+m}}-(\kappa _2,\boldsymbol
u)_{V_1'\oplus V_2'})\, d\boldsymbol u.
$$
for some continuous functions $\kappa _1=\kappa _1(x,y,\boldsymbol u)$
and $\kappa _2=\kappa _2(x,y,\boldsymbol u)$.

\par

We need to estimate $\nm {\mathcal K_{a,\omega ,Gv}}{L^\infty}$. By
Young's inequality and simple change of variables it follows that
\begin{equation*}
\nm {\mathcal K_{a,\omega ,Gv}}{L^\infty} \le \nm {Gv}{L^1}\cdot
J_{a,\omega},
\end{equation*}
where
\begin{multline*}
J_{a,\omega}=\essup {\mathsf X} \Big ( \int _{V_2'}\mathbf E_{a,\omega
}(x,y,\kappa _1(x,y,\boldsymbol u), (\kappa _2(x,y,\boldsymbol
u),\boldsymbol u)_{V_1'\oplus V_2'})\, d\boldsymbol u
\\[1ex]
\le \essup {\mathsf X} \Big ( \int _{V_2'}\Big (\sup _{\zeta
,\boldsymbol \tau} \mathbf E_{a,\omega }(x,y,\zeta , (\boldsymbol \tau
,\boldsymbol u)_{V_1'\oplus V_2'})\, d\boldsymbol u =\nmm a.
\end{multline*}
Hence
\begin{equation}\label{Kaomest}
\nm {\mathcal K_{a,\omega ,Gv}}{L^\infty} \le \nm {Gv}{L^1}\nmm a \le
C\exp (\nm {\fy ''}{M^{\infty ,1}_{(v)}})\nmm a .
\end{equation}

\par

A combination of \eqref{estagain4}, \eqref{Kaomest}, and the facts
that $\nm {F_1}{L^1}=\nm f{M^1_{(\omega _1)}}$ and $\nm {F_2}{L^1}=\nm
g{M^1_{(1/\omega _2)}}$ now gives that the pair $(a,\fy )$ is
admissible, and that (i) in Subsection \ref{ssec2.2} holds. The proof
is complete.
\end{proof}

\par

\begin{cor}\label{cor3.1A}
Assume that $N$, $\chi$, $\omega _0$, $\omega _j$, $v$ and $\fy$
for $j=1,2$ are the same as in Subsection \ref{ssec2.1}. Also assume
that $a\in \mathscr S'(\rr {N+m})$, and that one of the following
conditions holds:
\begin{enumerate}
\item $|\det (\fy ''_{\zeta ,\zeta })|\ge \mathsf d$ and $\nmm
a<\infty$, where
\begin{equation}\label{thm3.1Aineq1}
\nmm a = \sup _{x,y}\Big (\int _{\rr m}\sup _{\zeta ,\xi ,\eta }|V_\chi
a(X,\xi ,\eta ,z)\omega (X,\xi ,\eta ,z)|\, dz\Big ) \text ;
\end{equation}

\vrum

\item $m=n_1$, $|\det (\fy ''_{x,\zeta })|\ge \mathsf d$ and $\nmm
a<\infty$, where
\begin{equation}\label{thm3.1Aineq2}
\nmm a = \sup _{x,y}\Big (\int _{\rr {n_1}}\sup _{\zeta ,\eta
,z}|V_\chi a(X,\xi ,\eta ,z)\omega (X,\xi ,\eta ,z)|\, d\xi\Big )
\text ;
\end{equation}

\vrum

\item $m=n_2$, $|\det (\fy ''_{y,\zeta })|\ge \mathsf d$ and $\nmm
a<\infty$, where
\begin{equation}\label{thm3.1Aineq3}
\nmm a = \sup _{x,y}\Big (\int _{\rr {n_2}}\sup _{\zeta ,\xi
,z}|V_\chi a(X,\xi ,\eta ,z)\omega (X,\xi ,\eta ,z)|\, d\eta \Big ) .
\end{equation}
\end{enumerate}
Then the {\rm {(i)--(ii)}} in Subsection \ref{ssec2.2} hold for
$p=\infty$.
\end{cor}

\par

\begin{proof}
If (1) is fulfilled, then the result follows by choosing
\begin{equation*}
\begin{aligned}
V_1 &= V_1' = \sets {(\xi ,\eta ,0)\in \rr {N+m}}{\xi \in \rr {n_2}, \
\eta \in \rr {n_1}},
\\[1ex]
V_2 &=V_2' = \sets {(0,0,\zeta )\in \rr {N+m}}{\zeta \in \rr m},
\end{aligned}
\end{equation*}
$\boldsymbol \varrho =\zeta$, $\boldsymbol \tau =(\xi ,\eta )$ and
$\boldsymbol u=z$ in Theorem \ref{thm3.1A}. If instead (2) is
fulfilled, then the result follows by choosing
\begin{equation*}
\begin{aligned}
V_1 &= V_1' = \sets {(0,\eta ,z)\in \rr {N+m}}{\eta \in \rr {n_1},\
z\in \rr m},
\\[1ex]
V_2 &= V_2' = \sets {(\xi ,0,0 )\in \rr {N+m}}{\xi \in \rr {n_2}},
\end{aligned}
\end{equation*}
$\boldsymbol \varrho =x$, $\boldsymbol \tau =(\eta ,z)$ and
$\boldsymbol u=\xi$ in Theorem \ref{thm3.1A}. The result follows
similar arguments if instead (3) is fulfilled. The details are left
for the reader.
\end{proof}

\par

The set of all $a\in \mathscr S'(\rr {N+m})$ which fulfills that
$\nmm a<\infty$, where $\nmm \cdo$ is the same as in Theorem
\ref{thm3.1A} is neither a modulation space nor a coorbit space of
that type which is considered in Subsection \ref{ssec1.2}. However it
is still a coorbit spaces in the sense of \cite{FG1,FG2}.

\par

Next we  discuss  Fourier integral operators where the amplitudes
belong to coorbit spaces which are related to the amplitude space in
Theorem \ref{thm3.1A}.

\par

Assume that $a\in \mathscr S'(\rr {N+m})$, $\omega ,v\in \mathscr
P(\rr {2(N+m)})$, $\omega _j(\rr {2n_j})$ and $\fy \in C(\rr {N+m})$
satisfy $\fy ''\in M^{\infty ,1}_{(v)}(\rr {N+m})$ satisfy
\eqref{vvrel} and \eqref{weightsineq}, as before. Also assume that
$\omega _0\in \mathscr P(\rr{2N})$ satisfies
\begin{equation}\label{omega0cond}
\omega _0(x,y,\xi ,\eta )\le C\omega (X,\xi -\fy '_x(X),\eta -\fy
'_y(X),-\fy '_\zeta (X)).
\end{equation}
Roughly speaking, the main part of the analysis in Section \ref{sec2}
concerns of finding appropriate estimates of the function $\mathbf E$,
defined in \eqref{Efuncdef}.

\par

We need to make some further reformulations of the short-time Fourier
transform of the distribution kernel $K_{a,\fy}$ of $\op _\fy (a)$ in
terms of \eqref{fourrel}. Formally, the kernel can be written as
$$
K_{a,\fy}(x,y) = (2\pi )^{-N/2}\int _{\rr m} a(X)e^{i\fy (X)}\,
d\zeta .
$$
(Cf. Theorem \ref{fourop3}.) Hence, if $0\le \chi _j\in C_0^\infty
(\rr {n_j})$ for $j=1,2$ are the same as in Section \ref{sec2}, then
it follows by straight-forward computations that
\begin{equation}\label{eq4.0}
(V_{\chi _1\otimes \chi _2}K_{a,\fy})(x,y,\xi ,\eta ) = (\op _\fy
(a)(\chi _1(\cdo -y)e^{-i\scal \cdo \eta }), \chi _2(\cdo -x)e^{i\scal
\cdo \xi }).
\end{equation}
By letting $f=\chi _1(\cdo -y)e^{-i\scal \cdo \eta }$ and $g=\chi
_2(\cdo -x)e^{-i\scal \cdo \xi }$, it follows that
\begin{align}
|(V_{\chi _1}f(y_1,\eta _1)| &= |(V_{\chi _1}\chi _1)(y_1-y,\eta
_1+\eta )|\label{eq4.1}
\intertext{and}
|(V_{\chi _2}g(x_1,\xi _1)| &= |(V_{\chi _2}\chi _2)(x_1-x,\xi _1+\xi
)|.\label{eq4.2}
\end{align}
Now we choose $N_0$ large enough such that $\omega _0$ is moderate
with respect to $\eabs {\cdo }^{N_0}$, and we set
\begin{equation*}
F(\mathsf X) = |(V_{\chi _1}\chi _1)(y,-\eta )(V_{\chi _2}\chi
_2)(x,\xi )\eabs {\mathsf X}^{N_0}|.
\end{equation*}
Then $F$ is a continuous function which decreases rapidly to zera at
infinity. Furthermore, it follows from \eqref{eq4.1} and \eqref{eq4.2}
that
\begin{equation}\label{eq4.3}
|(V_{\chi _1}f(y_1,-\eta _1)(V_{\chi _2}g(x_1,\xi _1)\omega _0(\mathsf
X)| \le  CF(\mathsf X_1-\mathsf X)\omega (\mathsf X_1),
\end{equation}
where the first inequality follows from the fact that
$$
\omega _0(\mathsf X)\le C\omega _0(\mathsf X_1)\eabs {\mathsf
X-\mathsf X_1}^{N_0}.
$$

\par

By combining \eqref{vvrel}, \eqref{weightsineq}, \eqref{estfourop1}
and \eqref{omega0cond}--\eqref{eq4.3} we obtain
\begin{equation}\label{fundest1}
|(V_{\chi _1\otimes \chi _2}K_{a,\fy})(\mathsf X)\omega _0(\mathsf X
)| \le C\iint \mathcal R_{a,\omega ,\fy}(\mathsf X_1,\zeta
_1)F(\mathsf X_1-\mathsf X)\, d\zeta _1d\mathsf X_1,
\end{equation}
for some constant $C$.

\par

We have now the following parallel result of Theorem \ref{thm3.1A}.

\par

\begin{thm}\label{kernelest}
Assume that $N$, $\chi$, $\omega$, $\omega _j$ for $j=0,1,2$, $v$ and
$\fy$ are the same as in Subsection \ref{ssec2.1}.
Also assume that $p\in [1,\infty ]$, and that one of the following
conditions hold:
\begin{enumerate}
\item $a\in \mathscr S'(\rr {N+m})$ and $\nmm a<\infty$, where
\begin{equation*}
\nmm a = \Big (\iint _{\rr N}\Big (\int _{\rr m}\sup _z\Big (\iint
_{\rr N}|V_\chi a(X,\xi ,\eta ,z)\omega (X,\xi ,\eta ,z)|^p\, d\xi
d\eta \Big )^{1/p}\, d\zeta \Big )^p\, dxdy\Big )^{1/p}\text ;
\end{equation*}

\vrum

\item $|\det (\fy ''_{\zeta ,\zeta })|\ge \mathsf d$ for some $\mathsf
d>0$, $a\in \mathscr S'(\rr {N+m})$, and $\nmm a<\infty$, where
\begin{equation*}
\nmm a = \Big (\iint _{\rr N}\Big (\int _{\rr m}\sup _\zeta \Big
(\iint _{\rr N}|V_\chi a(X,\xi ,\eta ,z)\omega (X,\xi ,\eta ,z)|^p\,
d\xi d\eta \Big )^{1/p}\, dz \Big )^p\, dxdy\Big )^{1/p}.
\end{equation*}
\end{enumerate}
Then the {\rm{(i)--(iii)}} in Subsection \ref{ssec2.2} hold.
\end{thm}

\par

\begin{proof}
It suffices to prove (i). We only consider the case when (2) is
fulfilled. The other case follows by similar arguments and is left for
the reader.

\par

Let $G$ be the same as in the proof of
Theorem \ref{boulkemA}, $\mathcal Q_{a,\omega}$ and $\mathcal
R_{a,\omega}$ be as in \eqref{Qadef} and \eqref{Radef}, and let
$$
\mathbf E _{a,\omega }(X,\xi ,\eta ,z) = |V_\chi a(X,\xi ,\eta ,z)\omega (X,\xi
,\eta ,z)|.
$$
It follows from \eqref{fundest1} and H{\"o}lder's inequality that
\begin{align*}
|(V_{\chi _1\otimes \chi _2}K_{a,\fy})(\mathsf X)\omega _0(\mathsf X
)| &\le C\iint _{\rr {m+2N}}\big ( \mathcal R_{a,\omega ,\fy}(\mathsf
X_1,\zeta _1)F(\mathsf X_1-\mathsf X)^{1/p}\big ) \, F(\mathsf
X_1-\mathsf X)^{1/p'}\, d\zeta _1d\mathsf X_1
\\[1ex]
&\le C\nm F{L^1}^{1/p'}\Big (\int _{\rr {2N}} \Big (\int \mathcal
R_{a,\omega ,\fy}(\mathsf X_1,\zeta _1)\, d\zeta _1\Big )^pF(\mathsf
X_1-\mathsf X)\, d\mathsf X_1\Big )^{1/p},
\end{align*}
where $\nm F{L^1}$ is finite, since $F$ is rapidly decreasing to zero
at infinity. By letting $C_\fy =C\exp (C\nm {\fy ''}{M^{\infty
,1}_{(v)}})$ for some large constant $C$, and applying the $L^p$ norm
and Young's inequality, we get
\begin{equation}\label{estagain1}
\begin{aligned}
&\nm K{M^p_{(\omega _0)}}^p \le C_1 \int  _{\rr {2N}}\Big ( \int _{\rr
m}\mathcal R_{a,\omega ,\fy}(\mathsf X,\zeta )\, d\zeta \Big
)^pd\mathsf X
\\[1ex]
&\le C_1\nm G{L^1_{(v)}}^p\int _{\rr {2N}} \Big (\int _{\rr m}\mathcal
Q_{a,\omega ,\fy}(\mathsf X,\zeta )\, d\zeta \Big )^pd\mathsf X
\\[1ex]
&\le  C_\fy \int _{\rr {2N}} \Big (\int _{\rr m} \mathbf E_{a,\omega
}(X,\xi -\fy '_x(X),\eta -\fy '_y(X),-\fy '_\zeta (X))\, d\zeta \Big
)^pd\mathsf X.
\end{aligned}
\end{equation}
for some constant $C_1$. It follows now from Minkowski's
inequality that the latter integral can be estimated by
\begin{multline*}
\iint _{\rr N} \Big (\int _{\rr m}\Big (\iint _{\rr N}\mathbf
E_{a,\omega }(X,\xi -\fy '_x(X),\eta -\fy '_y(X),-\fy '_\zeta (X))^p\,
d\xi d\eta \Big )^{1/p}\, d\zeta \Big )^p dxdy
\\[1ex]
= \iint _{\rr N} \Big (\int _{\rr m} \Big (\iint _{\rr N}\mathbf
E_{a,\omega }(X,\xi ,\eta ,-\fy '_\zeta (X))^p\, d\xi d\eta \Big
)^{1/p}\, d\zeta \Big )^p dxdy .
\end{multline*}
By letting $C_\fy =C_2\exp (C_2\nm {\fy ''}{M^{\infty ,1}_{(v)}})$,
taking $\xi ,\eta ,-\fy '_\zeta (X),x,y$ as new variables of
integration, and using the fact that $|\det (\fy ''_{\zeta ,\zeta
})|\ge \mathsf d$, we get for some function $\kappa$ that
\begin{align*}
\nm K{M^p_{(\omega _0)}}^p &\le \frac {C_\fy}{\mathsf d} \iint _{\rr
N} \Big (\int _{\rr m}\Big (\iint _{\rr N}\mathbf E_{a,\omega
}(x,y,\kappa (x,y,z),\xi ,\eta ,z)^p\, d\xi d\eta \Big )^{1/p}\, dz
\Big )^p dxdy
\\[1ex]
&\le \frac {C_\fy}{\mathsf d} \iint _{\rr N} \Big (\int _{\rr m} \sup
_{\zeta}\Big (\iint _{\rr N} \mathbf E_{a,\omega }(x,y,\zeta ,\xi
,\eta ,z)^p\, d\xi d\eta \Big )^{1/p}\, dz \Big )^p dxdy
\\[1ex]
&= C_\fy \nmm a^p .
\end{align*}
This proves the assertion
\end{proof}

\par

We also have the following result parallel to Theorem \ref{kernelest}. 

\par

\begin{thm}\label{thm3.1B}
Assume that $N$, $\chi$, $\omega$, $\omega _j$, $v$, $\fy$, $V_j$,
$V_j'$, $\boldsymbol \varrho$ $\boldsymbol \tau$ and $\boldsymbol u$
for $j=0,1,2$ are the same as in Subsection \ref{ssec2.1}. Also assume
that $p\in [1,\infty]$, $a\in \mathscr S'(\rr {N+m})$
fulfills $\nmm a<\infty$, where
\begin{equation*}
\nmm a = \int _{V_2'}\Big ( \iint _{V_1\times V_1'} \Big ( \sup
_{\boldsymbol \varrho \in V_2}|V_\chi a(X,\xi ,\eta ,z)\omega (X,\xi
,\eta ,z)|\Big )^p\, d\boldsymbol t d\boldsymbol \tau \Big )^{1/p}\,
d\boldsymbol u \, ,
\end{equation*}
and that $|\det (\fy ''_{\boldsymbol \varrho ,\zeta})|\ge \mathsf
d$. Then {\rm{(i)--(iii)}} in Subsection \ref{ssec2.2} hold.
\end{thm}

\par

We note that the norm estimate on $a$ in Theorem \ref{thm3.1B} means
that $a\in \boldsymbol \Theta _{(\omega )}^{\mathsf p}(\overline V)$
with $\mathsf p=(\infty ,p,p,1)$ and $\overline
V=(V_2,V_1,V_1',V_2')$. The proof of Theorem \ref{thm3.1B} is based on
Theorem \ref{thm3.1A} and the following result which generalizes
Theorem \ref{kernelest} in the case $p=1$.

\par

\begin{prop}\label{kernelest11}
Assume that $N$, $\chi$, $\omega$, $\omega _j$, $v$, $\fy$, $V_j$,
$V_j'$, $\boldsymbol \varrho$ $\boldsymbol \tau$ and $\boldsymbol u$
for $j=0,1,2$ are the same as in Subsection \ref{ssec2.1}. Also assume
that $a\in \mathscr S'(\rr {N+m})$ satisfies $\nmm a <\infty$, where
\begin{equation*}
\nmm a = \iint _{\rr m\times V_1}\essup {\boldsymbol \varrho \in
V_2}\Big ( \iint _{\rr N} |V_\chi a(X,\xi ,\eta ,z)\omega (X,\xi ,\eta
,z)|\, d\xi d\eta \Big )\, d\boldsymbol tdz \, ,
\end{equation*}
and that $|\det (\fy ''_{\boldsymbol \varrho ,\zeta})|\ge \mathsf
d$. Then {\rm{(i)--(iii)}} in Subsection \ref{ssec2.2} hold for
$p=1$.
\end{prop}

\par

\begin{proof}
We use the same notations as in Subsection \ref{ssec2.1} and the proof
of Theorem \ref{boulkemA}. It follows from \eqref{estagain1} that
\begin{multline*}
\nm K{M^1_{(\omega _0)}} \le C_\fy \iint _{\rr {2N+m}}
\mathbf E_{a,\omega }(X,\xi -\fy '_x(X),\eta -\fy '_y(X),-\fy
'_\zeta (X))\, d\zeta d\mathsf X
\\[1ex]
=C_\fy \iint _{\rr {2N+m}}
\mathbf E_{a,\omega }(X,\xi ,\eta ,-\fy
'_\zeta (X))\, d\zeta d\mathsf X
\\[1ex]
=C_1C_\fy \iint _{V_1\times V_2}\Big ( \iint _{\rr N}\mathbf E_{a,\omega
}(X,\xi ,\eta ,-\fy '_\zeta (X))\, d\xi d\eta
\Big ) d\boldsymbol td\boldsymbol \varrho .
\end{multline*}
By taking $\boldsymbol t$ and $-\fy '_\zeta$ as new variables of
integration in the outer doubble integral, and using the fact that
$|\det (\fy ''_{\boldsymbol \varrho ,\zeta })|\ge \mathsf d$, we get
\begin{multline*}
\nm K{M^1_{(\omega _0)}} \le C_1C_\fy {\mathsf d}^{-1}\int _{\rr
m}\Big (\int _{V_1}\Big ( \iint _{\rr N} \mathbf E_{a,\omega }(X,\xi
,\eta ,z)\, d\xi d\eta \Big ) d\boldsymbol t\Big )dz
\\[1ex]
\le C_1C_\fy {\mathsf d}^{-1}\int _{\rr m}\Big (\int
_{V_1}\sup _{\boldsymbol \varrho \in V_2}\Big ( \iint _{\rr N}\mathbf
E_{a,\omega }(X,\xi ,\eta ,z)\, d\xi d\eta \Big ) d\boldsymbol t\Big
)dz
\\[1ex]
= C_1C_\fy {\mathsf d}^{-1}\nmm a.
\end{multline*}
This proves the result.
\end{proof}

\par

\begin{proof}[Proof of Theorem \ref{thm3.1B}.]
We start to consider the case $p=1$. By Proposition \ref{kernelest11}
(i), Minkowski's inequality and substitution of variables we obtain
\begin{multline*}
\nm K{M^1_{(\omega _0)}}\le C_\fy {\mathsf d}^{-1}\int _{\rr m}\Big
(\int _{V_1}\sup _{\boldsymbol \varrho \in V_2}\Big ( \iint _{\rr
N}\mathbf E_{a,\omega }(X,\xi ,\eta ,z)\, d\xi d\eta \Big )
d\boldsymbol t\Big )dz
\\[1ex]
\le C_\fy {\mathsf d}^{-1}\iiiint _{V_1\times \rr {N+m}}\essup
{\boldsymbol \varrho \in V_2} \mathbf E_{a,\omega }(X,\xi ,\eta ,z)\,
d\boldsymbol t d\xi d\eta dz
\\[1ex]
= C_1C_\fy \int _{V_2'}\Big (\iint _{V_1\times V_1'} \essup
{\boldsymbol \varrho \in V_2} \mathbf E_{a,\omega }(X,\xi ,\eta ,z) \,
d\boldsymbol t d\boldsymbol \tau \Big )\, d\boldsymbol u ,
\end{multline*}
for some constant $C_1$, and the result follows in this case.

\par

Next we consider the case $p=\infty$. By Theorem \ref{thm3.1A} we get
\begin{multline*}
\nm K{M^\infty _{(\omega _0)}}\le C_\fy \sup _{x,y}\Big ( \int _{V_2'}
\sup _{\zeta ,\boldsymbol \tau} \big (\mathbf E_{a,\omega }(X,\xi
,\eta ,z)\big )\, d\boldsymbol u\Big )
\\[1ex]
\le C_\fy \int _{V_2'} \Big (\essup {({\boldsymbol t},{\boldsymbol
\tau})\in V_1\times V_1'}\big ( \sup _{\boldsymbol \varrho \in V_2}
\mathbf E_{a,\omega }(X,\xi ,\eta ,z)\, \big )\Big )\, d\boldsymbol u ,
\end{multline*}
and the result follows in this case as well.

\par

The theorem now follows for general $p$ by interpolation, using
Proposition \ref {coorbintepol}. The proof is complete.
\end{proof}

\par

By interpolating Theorem \ref{boulkemA} and Theorem \ref{thm3.1B} we
get the following result.

\par

\begin{thm}\label{thm3.1C}
Assume that $N$, $\chi$, $\omega$, $\omega _j$, $v$, $\fy$, $V_j$,
$V_j'$, $\boldsymbol \varrho$ $\boldsymbol \tau$ and $\boldsymbol u$
for $j=0,1,2$ are the same as in Subsection \ref{ssec2.1}. Also assume
that $p,q\in [1,\infty]$, $a\in \mathscr
S'(\rr {N+m})$ fulfills $\nmm a<\infty$, where
\begin{equation*}
\nmm a = \int _{V_2'}\Big ( \int _{V_1'} \Big ( \int _{V_1} \Big (
\sup _{\boldsymbol \varrho \in V_2}|V_\chi a(X,\xi ,\eta ,z)\omega
(X,\xi ,\eta ,z)|\Big )^p\, d\boldsymbol t \Big )^{q/p} d\boldsymbol
\tau \Big )^{1/q}\, d\boldsymbol u \, ,
\end{equation*}
and that in addition $n_1=n_2$ and \eqref{detphicond} and $|\det (\fy ''_{\boldsymbol \varrho ,\zeta})|\ge \mathsf d$ hold for some $\mathsf d>0$. Then the following is true:
\begin{enumerate}
\item if $p'\le q\le p$ and $p_1,p_2\in [1,\infty ]$ satisfy
\begin{equation}\label{pjqcond}
q\le p_1',p_2 \le p,\qquad \text{and}\quad \frac 1{p_1'}+\frac 1{p_2} =\frac 1p+\frac 1q,
\end{equation}
with strict inequalities in \eqref{pjqcond}, then the definition of $\op _\fy (a)$ extends uniquely to a continuous map from $M^{p_1}_{(\omega _1)}$ to $M^{p_2}_{(\omega _2)}$;

\vrum

\item if $q\le \min (p,p')$, then $\op _\fy (a)\in
\mathscr I_p(M^2_{(\omega _1)},M^2_{(\omega _2)})$.
\end{enumerate}
\end{thm}

\par

We note that the norm estimate on $a$ in Theorem \ref{thm3.1B} means
that $a\in \boldsymbol \Theta _{(\omega )}^{\mathsf p}(\overline V)$
with $\mathsf p=(\infty ,p,q,1)$ and $\overline
V=(V_2,V_1,V_1',V_2')$. 

\par

\begin{proof}
In order to prove (1) we note that the result holds when $(p,q)=(\infty ,1)$ or $q=p$, in view of Theorems \ref{boulkemA} and \ref{thm3.1B}. Next assume that $q=p'$ for $p\ge 2$, and set $\mathsf p_1=(\infty ,\infty,1,1)$ and $\mathsf p=(\infty ,2,2,1)$. Then it follows from Theorem \ref{boulkemA} and \ref{thm3.1B} that the bilinear form
$$
T(a,f) \equiv \op _\fy (a)f
$$
is continuous from
\begin{alignat*}{3}
&\boldsymbol \Theta _{(\omega )}^{\mathsf p_1}\times M^{p}_{(\omega _1)}&\quad &\text{to}&\quad  &M^{p}_{(\omega _2)},\quad 1<p<\infty,
\intertext{and from}
&\boldsymbol \Theta _{(\omega )}^{\mathsf p_2}\times M^{2}_{(\omega _1)} &\quad &\text{to}&\quad  &M^{2}_{(1/\omega _2)}.
\end{alignat*}
By interpolation, using Theorem 4.4.1 in \cite{BL}, Proposition \ref{interpolmod} and Proposition \ref{coorbintepol}, it follows that if $q=p'<2$, then $T$ extends uniquely to a continuous map from
$$
\boldsymbol \Theta _{(\omega )}^{\mathsf p}\times M^{p_1}_{(\omega _1)}\quad \text{to}\quad  M^{p_2}_{(\omega _2)},
$$
when $p'<p_1=p_2<p$. This proves (1) when $q=p$ or $q=p'$.

\par

For $q\in (p',p)$, the result now follows by interpolation between the case $q=p'$ and $p_1=p_2=p_0$ where $p'<p_0<p$, and the case $q=p$ and $p_1'=p_2=p$. In fact, by interpolation it follows that $T$ extens to a continuous map from
$$
\boldsymbol \Theta _{(\omega )}^{\mathsf p}\times M^{p_1}_{(\omega _1)}\quad \text{to}\quad  M^{p_2}_{(\omega _2)}
$$
when
$$
\frac 1q=\frac {1-\theta}{p'}+\frac {\theta}{p},\quad \frac 1{p_1}=\frac {1-\theta}{p_0}+\frac {\theta}{p'},\quad \frac 1{p_2}=\frac {1-\theta}{p_0}+\frac {\theta}{p}.
$$
It is now straight-forward to control that these conditions are equivalent with those conditions in (1), and the assertion follows.

\par

In order to prove (2), it is no restriction to assume that $q=\min (p,p')$. If $p=\infty$ and $q=1$, then the
result is a consequence of Theorem \ref{boulkemA}. If instead $1\le
q=p\le 2$, then the result follows from Theorem \ref{thm3.1B}. The
remaining case $2\le p=q'\le \infty$ now follows by
interpolation between the cases $(p,q)=(2,2)$ and $(p,q)=(\infty ,1)$,
using \eqref{interpschatt} or \eqref{realinterpschatt}, and the interpolation properties in
Section \ref{ssec1.2}. The proof is complete.
\end{proof}

\par

\section{Consequences}\label{sec3}

\par

In this section we list some consequences of the results in Section
\ref{sec2}. In Subsection \ref{ssec3.1} we consider Fourier integral
operators where the amplitudes depend on two variables only. In
Subsection \ref{ssec3.2} we consider Fourier integral operators with
smooth amplitudes.

\par

\subsection{Fourier integral operators with amplitudes depending on
two variables}\label{ssec3.1}
We start to discuss Schatten-von Neumann operators for Fourier
integral operators with symbols in $M^{p,q}_{(\omega)}(\rr{2n})$ and
phase functions in $M^{\infty,1}_{(v)}(\rr {3n})$, for appropriate
weight functions $\omega$ and $v$. We assume here that the phase
functions depend on $x, y, \zeta \in \rr n$ and that the amplitudes
only depend on the $x$ and $\zeta$ variables and are
independent of the $y$ variable. Note that here we have assumed that
the numbers $n_1$, $n_2$ and $m$ in Section \ref{sec2} are equal to
$n$. As in the preivous section, we use the notation $X, Y, Z,\dots$
for tripples of the form $(x,y,\zeta )\in \rr {3n}$.

\par

The first aim is to establish a weighted version of Theorem 2.5 in
\cite {CT2}. To this purpose, we need to transfer the conditions for
the weight and phase functions from Section \ref{sec2}. Namely here
and in the following we assume that $\fy \in C(\rr {3n})$, $\omega
_0,\omega \in \mathscr P(\rr {4n})$, $v_1\in \mathscr P(\rr n)$,
$v_2\in \mathscr P(\rr {2n})$ and $v\in \mathscr P(\rr {6n})$. A
condition on the phase function is
\begin{equation}\label{phasecond}
|\det(\fy ''_{y,\zeta }(X))|\ge \mathsf d,\qquad X=(x,y,\zeta )\in
\rr {2n}
\end{equation}
for some constant $\mathsf d>0$, and the conditions in
\eqref{weightsineq} in Subsection \ref{ssec2.1} are modified into:
\begin{equation}\tag*{(\ref{weightsineq})$'$}
\begin{aligned}
\omega _0(x,y,\xi  ,\varphi '_y(X)) &\le C\omega (x,\zeta,\xi-\varphi
'_x(X),-\varphi '_\zeta (X)),
\\[1ex]
\frac {\omega _2(x,\xi )}{\omega _1(y,-\eta )} &\le C\omega _0(x,y,\xi
, \eta),
\\[1ex]
\omega _0(x,y,\xi , \eta _1+\eta _2) &\le  C\omega _0(x,y,\xi ,\eta _1
)v_1(\eta _2),
\\[1ex]
\omega (x,\zeta ,\xi _1+\xi _2,z_1+z_2) &\le \omega (x,\zeta ,\xi
_1,z_1)v_2(\xi _2,z_2),
\\[1ex]
v(X,\xi ,\eta ,z) &= v_1(\eta )v_2(\xi ,z),\qquad x,y,z,
z_j,\xi ,\xi _j ,\eta ,\zeta \in \rr n.
\end{aligned}
\end{equation}

\par

For conveniency we also set $\op _{1,0,\fy} (a) =\op _{\fy}(a_1)$ when $a_1(x,y,\zeta )=a(x,\zeta )$.

\par

\begin{thm}\label{fourop3}
Assume that $p\in [1,\infty ]$, $\mathsf d >0$, $v\in \mathscr P(\rr
{6n})$ is submultiplicative, $\omega _0,\omega \in \mathscr P(\rr
{4n})$ and that $\varphi \in C(\mathbf {R}^{3n})$ are such that $\fy$ is
real-valued, $\varphi ^{(\alpha )}\in M ^{\infty,1}_{(v)}$ for all
multi-indices $\alpha$ such that $|\alpha | =2$, and \eqref{phasecond}
and \eqref{weightsineq}$'$ are fulfilled for some constant $C$. Then the following is true:
\begin{enumerate}
\item the map
$$
a\mapsto K_{a,\fy}(x,y)\equiv \int a(x, \zeta )e^{i\varphi (x,y,\zeta )}\,
d\zeta ,
$$
from $\mathscr S(\rr {2n})$ to $\mathscr S'(\rr {2n})$ extends
uniquely to a continuous map from $M^p_{(\omega )}(\rr {2n})$ to
$M^p_{(\omega _0)}(\rr {2n})$;

\vrum

\item the map $a\mapsto \op _{1,0,\fy}(a)$ from $\mathscr S(\rr {2n})$ to $\mathscr L(\mathscr S(\rr n),\mathscr S'(\rr n))$ extends uniquely to a continuous map from $M^p_{(\omega )}(\rr {2n})$ to $\mathscr L(\mathscr S(\rr n),\mathscr S'(\rr n))$;

\vrum

\item if $a\in M^p_{(\omega )}(\rr {2n})$, then the definition of $\op _{1,0,\fy} (a)$ extends
uniquely to a continuous operator from $M^{p'}_{(\omega _1)}(\rr
{n})$ to $M^{p}_{(\omega _2)}(\rr {n})$. Furthermore, for some
constant $C$ it holds
$$
\nm {\op _{1,0,\fy} (a)}{M^{p'}_{(\omega _1)}\to M^{p}_{(\omega _2)}} \le
C{\mathsf d}^{-1}\exp (\nm {\fy ''}{M^{\infty ,1}_{(v)}}) \nmm a \,  \text ;
$$

\vrum

\item if $a\in M^{\infty ,1}_{(\omega )}(\rr {2n})$, then the definition of $\op _{1,0,\fy} (a)$ from $\mathscr S(\rr n)$ to $\mathscr S'(\rr n)$ extends uniquely to a continuous operator from $M^{p}_{(\omega
_1)}$ to $M^p_{(\omega _2)}$;

\vrum

\item if $q\le \min (p,p')$, $a\in M^{p,q}_{(\omega )}(\rr {2n})$, and in addition \eqref{detphicond} holds, then $\op _{1,0,\fy} (a)\in \mathscr I_p(M^2_{(\omega _1)},M^2_{(\omega _2)})$.
\end{enumerate}
\end{thm}

\par

\begin{proof}
We start to prove the continuity assertions. Let $a_1(x,y,\zeta
)=a(x,\zeta )$, and let
$$
\widetilde \omega (x,y,\zeta ,\xi ,\eta ,z) =\omega (x,\zeta ,\xi
,z)v_1(\eta ).
$$
By Proposition \ref{extops} it follows that
prove that $a_1\in \boldsymbol \Theta _{(\widetilde \omega )}^{\mathsf
p}(\overline V)$ with $\mathsf p=(\infty ,p,p,1)$ and $\overline
V=(V_2,V_1,V_1',V_2')$. Hence Theorem \ref{thm3.1B} shows that it
suffices to prove that \eqref{weightsineq} holds after $\omega$ has
been replaced by $\widetilde \omega$.

\par

By \eqref{weightsineq}$'$ we have
\begin{multline*}
\omega _0(x,y,\xi ,\eta )\le C\omega _0(x,y,\xi ,\fy '_y(X))v_1(\eta
-\fy '_y(X))
\\[1ex]
\le C^2\omega (x,\zeta ,\xi -\fy '(X),-\fy '_\zeta (X))v_1(\eta -\fy
'_y(X))
\\[1ex]
=C^2\widetilde \omega (x,y,\zeta ,\xi -\fy '(X),\eta -\fy
'_y(X),-\fy '_\zeta (X)).
\end{multline*}
This proves that the first two inequalities in \eqref{weightsineq}
hold. Furthermore, since $v_1$ is submultiplicative we have
\begin{multline*}
\widetilde \omega (X,\xi _1+\xi _2,\eta _1+\eta _2,z_1+z_2) = \omega
(x,\zeta ,\xi _1+\xi _2,z_1+z_2)v_1(\eta _1+\eta _2)
\\[1ex]
\le C\omega (x,\zeta ,\xi _1,z_1)v_2(\xi _2,z _2)v_1(\eta _1)v(\eta
_2) = C\widetilde \omega (X,\xi _1,\eta _1,z_1)v(\xi _2,\eta _2,z_2),
\end{multline*}
for some constant $C$. This proves the last inequality in
\eqref{weightsineq}, and the continuity assertions follow.

\par

It remains to prove the uniqueness. If $p<\infty$, then the uniqueness
follows from the fact that $\mathscr S$ is dense in
$M^p_{(\omega)}$.

\par

Next we consider the case $p=\infty$. Assume that $a\in M^1_{(\omega
)}(\rr {2n})$ and $b\in M^1_{(1/\omega _0)}(\rr
{2n})$, and let $\widetilde \fy (x,y,\xi )=-\fy (x,\xi ,y)$. Since
\eqref{phasecond} also holds when $\fy$ is replaced by $\widetilde
\fy$, the first part of the proof shows that $K_{b,\widetilde
\fy}\in M^{1}_{(1/\omega )}$. Furthermore, by straight-forward
computations we have
\begin{equation}\label{Kadjoint}
(K_{a,\fy},b)=(a,K_{b,\widetilde \fy}).
\end{equation}
In view of Proposition \ref{p1.4} (3), it follows that the right-hand
side in \eqref{Kadjoint} makes sense if, more generally, $a$ is an
arbitrary element in $M^{\infty}_{(\omega )}(\rr {2n})$, and then
$$
|(a,K_{b,\widetilde \fy})|\le C\mathsf d^{-1}\nm a{M^\infty _{(\omega
)}}\nm b{M^1_{(1/\omega _0)}}\exp (C\nm {\fy ''}{M^{\infty,1}_{(v)}}),
$$
for some constant $C$ which is independent of $\mathsf d$, $a\in
M^\infty _{(\omega )}$ and $b\in M^1_{(1/\omega _0)}$.

\par

Hence, by letting $K_{a,\fy}$ be defined as
\eqref{Kadjoint} when $a\in M^\infty$, it follows that $a\mapsto
K_{a,\fy}$ on $M^1$ extends to a continuous map on
$M^\infty$. Furthermore, since $\mathscr S$ is dense in $M^\infty$
with respect to the weak$^*$ topology, it follows that this extension
is unique. We have therefore proved the theorem for $p\in \{ 1,\infty
\}$.
\end{proof}

\par

Finally we remark that the results in Section \ref{sec2} also give Theorem \ref{fourop3}$'$, which concerns Fourier integral operators of the form
$$
\op _{t_1,t_2,\fy}(a)f(x)\equiv \iint a(t_1x+t_2y,\xi
)f(y)e^{i\varphi (t_1x+t_2y,-t_2x+t_1y,\xi )}\, dyd\xi .
$$
It is then natural to assume that the conditions
\eqref{phasecond}$'$ is replaced by
\begin{equation}\tag*{(\ref{phasecond})$'$}
 t_1^2+t_2^2=1,\qquad |\det(\fy ''_{y,\xi }(X))|\ge \mathsf d,
\end{equation}
and
\begin{equation}\tag*{(\ref{weightsineq})$''$}
\begin{aligned}
&\omega _0(t_1x+t_2y,-t_2x+t_1y,t_1\xi+t_2\varphi '_y(X),-t_2\xi+t_1\varphi '_y(X))
\\[1ex]
&\le C \omega (x,\zeta ,\xi -\varphi '_x(X),-\varphi '_\zeta (X))
\\[1ex]
\frac {\omega _2(x,\xi )}{\omega _1(y,-\eta )} &\le C\omega _0(x,y,\xi
, \eta),
\\[1ex]
\omega _0(x,y,\xi+t_2\eta _2, \eta _1+t_1\eta _2) &\le  \omega _0(x,y,\xi
,\eta _1)v_1(\eta _2) 
\\[1ex]
\omega (x,\zeta ,\xi _1+\xi _2,z_1+z_2) &\le \omega (x,\zeta ,\xi
_1,z_1)v_2(\xi _2,z_2),
\\[1ex]
v(X,\xi ,\eta ,z) &= v_1(\eta )v_2(\xi ,z),\qquad x,y,z,
z_j,\xi ,\xi _j ,\eta ,\zeta \in \rr n.
\end{aligned}
\end{equation}

\par

\renewcommand{\rubrik}{Theorem \ref{fourop3}$'$}

\par

\begin{tom}
Assume that $p\in [1,\infty ]$, $\mathsf d >0$, $v\in \mathscr P(\rr
{6n})$ is submultiplicative, $\omega _0,\omega \in \mathscr P(\rr
{4n})$ and that $\varphi \in C(\mathbf {R}^{3n})$ are such that $\fy$ is
real-valued, $\varphi ^{(\alpha )}\in M ^{\infty,1}_{(v)}$ for all
multi-indices $\alpha$ such that $|\alpha | =2$, and \eqref{phasecond}$'$
and \eqref{weightsineq}$''$ are fulfilled for some constants $t_1$, $t_2$ and $C$. Then the following is true:
\begin{enumerate}
\item the map
$$
a\mapsto K_{a,\fy}(x,y)\equiv \int a(t_1x+t_2y,\zeta
)e^{i\varphi (t_1x+t_2y,-t_2x+t_1y,\zeta )}\, d\zeta ,
$$
from $\mathscr S(\rr {2n})$ to $\mathscr S'(\rr {2n})$ extends
uniquely to a continuous map from $M^p_{(\omega )}(\rr {2n})$ to
$M^p_{(\omega _0)}(\rr {2n})$;

\vrum

\item the map $a\mapsto \op _{t_1,t_2,\fy}(a)$ from $\mathscr S(\rr {2n})$ to $\mathscr L(\mathscr S(\rr n),\mathscr S'(\rr n))$ extends uniquely to a continuous map from $M^p_{(\omega )}(\rr {2n})$ to $\mathscr L(\mathscr S(\rr n),\mathscr S'(\rr n))$;

\vrum

\item if $a\in M^p_{(\omega )}(\rr {2n})$, then the definition of $\op _{t_1,t_2,\fy} (a)$ extends
uniquely to a continuous operator from $M^{p'}_{(\omega _1)}(\rr
{n})$ to $M^{p}_{(\omega _2)}(\rr {n})$. Furthermore, for some
constant $C$ it holds
$$
\nm {\op _{t_1,t_2,\fy} (a)}{M^{p'}_{(\omega _1)}\to M^{p}_{(\omega _2)}} \le
C{\mathsf d}^{-1}\exp (\nm {\fy ''}{M^{\infty ,1}_{(v)}}) \nmm a \,  \text ;
$$

\vrum

\item if $a\in M^{\infty ,1}_{(\omega )}(\rr {2n})$, then the definition of $\op _{t_1,t_2,\fy} (a)$ from $\mathscr S(\rr n)$ to $\mathscr S'(\rr n)$ extends uniquely to a continuous operator from $M^{p}_{(\omega _1)}$ to $M^p_{(\omega _2)}$;

\vrum

\item if $q\le \min (p,p')$, $a\in M^{p,q}_{(\omega )}(\rr {2n})$, and in addition \eqref{detphicond} holds, then $\op _{t_1,t_2,\fy} (a)\in \mathscr I_p(M^2_{(\omega _1)},M^2_{(\omega _2)})$.
\end{enumerate}
\end{tom}

\par

\begin{proof}
By letting
$$
x_1 =t_1x+t_2y,\quad y_1= -t_2x+t_1y
$$
as new coordinates, it follows that we may assume that $t_1=1$ and
$t_2=0$, and then the result agrees with Theorem
\ref{fourop3}. The proof is complete.
\end{proof}

\par

\subsection{Fourier integral operators with smooth amplitudes}\label{ssec3.2}
Next we apply Theorem \ref{thm3.1C} to Fourier integral operators with
smooth amplitudes. We recall that the condition on $a$ in Theorem
\ref{thm3.1C} means exactly that  $a\in \boldsymbol \Theta _{(\omega
)}^{\mathsf p}(\overline V)$ with $\mathsf p=(\infty ,p,q,1)$ and
$\overline V=(V_2,V_1,V_1',V_2')$. In what follows we consider the
case when $n_1=n_2=m=n$ and
\begin{equation}\label{vectspcond}
V_1=V_1'= \sets {(x,0,\zeta )\in \rr {3n}}{x,\zeta \in \rr n},\quad
\text{and}\quad V_2 = V_2' =  \sets {(0,y,0 )\in \rr {3n}}{y \in \rr
n}.
\end{equation}
However, the analysis presented here also holds without these
restrictions. The details are left for the reader. We are especially
concerned with spaces of amplitudes of the form
$$
C_{(\omega )}^{N,p}(\rr {3n})=\sets {a\in C^N(\rr {3n})}{\nm
a{C_{(\omega )}^{N,p}}<\infty},
$$
where $N\ge$ is an integer, $\omega \in \mathscr P(\rr {3n})$ and
$$
\nm a{C_{(\omega )}^{N,p}} \equiv \sum _{|\alpha |\le N}\Big ( \iint
_{\rr {2n}} \, \nm {a(x,\cdo ,\zeta )\omega (x,\cdo ,\zeta
)}{L^\infty}^p\, dxd\zeta \Big )^{1/p}.
$$
We also set
$$
C_{(\omega )}^{\infty ,p}(\rr {3n}) =\cap _{N\ge 0}C_{(\omega
)}^{N,p}(\rr {3n})
$$

\par

The following proposition links $C_{(\omega )}^{N,p}$ with
$\boldsymbol \Theta _{(\omega )}^{\mathsf p}(\overline V)$:

\par

\begin{prop}\label{identities2}
Assume that \eqref{vectspcond} is fulfilled, $N\ge 0$ is an integer,
$\overline V=(V_2,V_1,V_1',V_2)$, $\mathsf p=(\infty ,p,q,1)$,
$\mathsf p_1=(\infty ,p,1,1)$ and that $\mathsf p_2=(\infty ,p,\infty
,\infty )$. Also assume that $\omega \in \mathscr P(\rr {3n})$, and
let
$$
\omega _s(X,\xi ,\eta ,z)=\omega (X)\eabs {\xi ,\eta ,z}^s,\quad s\in
\mathbf R .
$$
If $s_1<-2n/q'$ when $q>1$ and $s_1\le 0$ when $q=1$, and
$s_2>n(q+2)/q$, then the following embedding holds:
\begin{align}
\boldsymbol \Theta _{(\omega _N)}^{\mathsf p_1} &\hookrightarrow
\boldsymbol \Theta _{(\omega _N)}^{\mathsf p} \hookrightarrow
\boldsymbol \Theta _{(\omega _N)}^{\mathsf p_2}\label{embcoorb1}
\\[1ex]
\boldsymbol \Theta _{(\omega _{N+s_2})}^{\mathsf p_2} &\hookrightarrow
\boldsymbol \Theta _{(\omega _N)}^{\mathsf p} \hookrightarrow
\boldsymbol \Theta _{(\omega _{N+s_1})}^{\mathsf p_1}\label{embcoorb2}
\intertext{and}
C_{(\omega )}^{N+3n+1,p}  &\hookrightarrow \boldsymbol \Theta _{(\omega
_N)}^{\mathsf p_1}\hookrightarrow  C_{(\omega
)}^{N,p}.\label{embcoorb3}
\end{align}
\end{prop}

\par

For the proof it is convenient to let $\mathscr P_0(\rr n)$ be the set
of all $\omega \in \mathscr P(\rr n)\cap C^{\infty}(\rr n)$ such that
$\omega ^{(\alpha )}/\omega$ is bounded for all multi-indices
$\alpha$.

\par

\begin{lemma}\label{lemma1identities2}
Assume that $\mathsf p = (p,q,r,s)\in [1,\infty ]^4$, and that $N\ge
0$ is an integer. Then the following is true:
\begin{enumerate}
\item if $\omega \in \mathscr P(\rr n)$, then it exists an element
$\omega _0\in \mathscr P_0(\rr n)$ such that
\begin{equation}\label{weightequiv}
C^{-1}\omega _0 \le \omega \le C\omega _0,
\end{equation}
for some constant $C$;

\vrum

\item if $\omega \in \mathscr P(\rr {2n})$, $\widetilde \omega _j\in
\mathscr P_0(\rr {2n})$ for $j=1,2$ are such that $\widetilde \omega
_1(x,\xi )=\widetilde \omega _1(x)$ and $\widetilde \omega 2(x,\xi
)=\widetilde \omega _2(\xi )$, then the mappings
\begin{align*}
f &\mapsto \widetilde \omega _1\cdot f
\intertext{and}
f &\mapsto \widetilde \omega _2(D) f
\end{align*}
are homeomorphisms from $\boldsymbol \Theta _{(\widetilde \omega _1\,
\omega )}^{\mathsf p}(\overline V)$ and from $\boldsymbol \Theta
_{(\widetilde \omega _2\, \omega )}^{\mathsf p}(\overline V)$
respectively to $\boldsymbol \Theta _{(\omega )}^{\mathsf p}(\overline
V)$. Furthermore, if $\omega _{N_1,N_2}(x,\xi )=\omega (x,\xi )\eabs
x^{N_2}\eabs \xi ^{N_1}$, then
\begin{equation}\label{dercoorb}
\begin{aligned}
&\boldsymbol \Theta _{(\omega _{N_1,N_2})}^{\mathsf p}(\overline V) =
\sets {f\in \mathscr S'(\rr n)}{x^\alpha \partial ^{\beta }f\in
\boldsymbol \Theta _{(\omega )}^{\mathsf p}(\overline V),\ |\alpha
|\le N_2,\ |\beta|\le N_1}
\\[1ex]
&= \sets {f\in \mathscr S'(\rr n)}{f,\, x_j^{N_2}f,\, D_k^{N_1}f,\,
x_j^{N_2}D_k^{N_1}f\in \boldsymbol \Theta _{(\omega )}^{\mathsf
p}(\overline V),\ 1\le j,k\le n}\text ;
\end{aligned}
\end{equation}

\vrum

\item if $\omega \in \mathscr P(\rr {6n})$ and $\omega _0\in \mathscr
P_0(\rr {6n})$ are such that $\omega (X,\xi ,\eta ,z)=\omega (X)$ and
$\omega _0(X,\xi ,\eta ,z)=\omega _0(X)$, then the map $a\mapsto
\omega _0\cdot a$ is a bijection from $C _{(\omega _0\omega )}^{N,
p}(\rr {3n})$ to $C _{(\omega )}^{N, p}(\rr {3n})$.
\end{enumerate}
\end{lemma}

\par

\begin{proof}
The assertion (1) follows from Lemma 1.2 in \cite{To9} The first part
of assertion (2) is a consequence of Theorem 3.2 when $\boldsymbol
\Theta _{(\omega )}^{\mathsf p}(\overline V)$ is a modulation
space. The general case follows by similar arguments as in the proof
of that theorem. We omit the details. The assertion (3) is a
straight-forward consequence of the definitions.

\par

It remains to prove \eqref{dercoorb}. It is convenient to set
$$
\sigma _{N_1,N_2}(x,\xi )=\eabs x^{N_2}\eabs \xi ^{N_1}.
$$
Furthermore, let $M_0$  be the set of all $f\in {\boldsymbol \Theta}
_{(\omega )}^{\mathsf p}$ such that $x^{\beta}\partial ^\alpha f\in
{\boldsymbol \Theta} _{(\omega )}^{\mathsf p}$ when $|\alpha |\le N_1$
and $|\beta |\le N_2$, and let $\widetilde M_0$ be the set of all
$f\in {\boldsymbol \Theta} _{(\omega )}^{\mathsf p}$ such that
$x_j^{N_2}\partial _k^{N_1} f\in {\boldsymbol \Theta} _{(\omega
)}^{\mathsf p}$ for $j,k=1,\dots , N$. We shall prove that
$M_0=\widetilde M_0={\boldsymbol \Theta} ^{\mathsf p}_{(\sigma
_{N_1,N_2}\omega )}$. Obviously, $M_0\subseteq \widetilde M_0$.  By
the first part of (2) it follows that ${\boldsymbol \Theta} ^{\mathsf
p}_{(\sigma _{N_1,N_2}\omega )}\subseteq M_0$. The result therefore
follows if it is proved that $\widetilde M_0\subseteq {\boldsymbol
\Theta} ^{\mathsf p}_{(\sigma _{N_1,N_2}\omega )}$.

\par

In order to prove this, assume first that $N_1=N$, $N_2=0$, 
$f\in \widetilde M_0$, and choose open sets
$$
\Omega _0=\sets{\xi \in \rr m}{|\xi |<2},\quad \text{and}\quad \Omega
_j=\sets {\xi \in \rr m}{1<|\xi |<n|\xi _j|}.
$$
Then $\cup _{j=0}^n\Omega _j=\rr n$, and there are non-negative
functions $\fy _0,\dots ,\fy _n$ in $S^0_0$ such that $\supp \fy _j
\subseteq \Omega _j$ and $\sum _{j=0}^n\fy _j=1$. In particular,
$f=\sum _{j=0}^nf_j$ when  $f_j=\fy _j(D)f$. The result follows if we
prove that $f_j\in {\boldsymbol \Theta} ^{\mathsf p}_{(\sigma
_{N,0}\omega )}$ for every $j$.

\par

Now set $\psi _0(\xi )=\sigma _N(\xi )\fy_0(\xi )$ and $\psi _j(\xi
)=\xi _j^{-N}\sigma _N(\xi )\fy _j(\xi )$ when $j=1,\dots ,n$. Then
$\psi _j\in S^0_0$ for every $j$. Hence the first part of (2) shows
that gives
\begin{multline*}
\nm {f_j}{{\boldsymbol \Theta} ^{\mathsf p}_{(\sigma
_{N,0}\omega)}}\le C_1\nm {\sigma _N(D)f_j}{{\boldsymbol \Theta}
_{(\omega )}^{\mathsf p}}
\\[1ex]
=C_1\nm {\psi _j(D)\partial _j^Nf}{{\boldsymbol \Theta} _{(\omega
)}^{\mathsf p}}\le C_2\nm {\partial _j^Nf}{{\boldsymbol \Theta}
_{(\omega )}^{\mathsf p}}<\infty
\end{multline*}
and
$$
\nm {f_0}{{\boldsymbol \Theta} ^{\mathsf p}_{(\sigma
_{N,0}\omega)}}\le C_1\nm {\sigma _N(D)f_0}{{\boldsymbol \Theta}
_{(\omega )}^{\mathsf p}} =C_1\nm {\psi _0(D)f}{{\boldsymbol \Theta}
^{\mathsf p}_{(\omega )}}\le C_2\nm {f}{{\boldsymbol \Theta} _{(\omega
)}^{\mathsf p}}<\infty
$$
for some constants $C_1$ and $C_2$. This proves that
\begin{equation}\label{uppskattning1}
\nm f{{\boldsymbol \Theta} ^{\mathsf p}_{(\sigma _{N,0}\omega )}} \le
C\Big ( \nm f{{\boldsymbol \Theta} _{(\omega )}^{\mathsf p}} +\sum
_{j=1}^N \nm {\partial _j^Nf}{{\boldsymbol \Theta} _{(\omega
)}^{\mathsf p}}\Big ),
\end{equation}
and the result follows in this case.

\par

If we instead split up $f$ into $\sum \fy _jf$, then similar arguments
show that
\begin{equation}\label{uppskattning2}
\nm f{{\boldsymbol \Theta} ^{\mathsf p}_{(\sigma _{0,N}\omega )}} \le
C\Big ( \nm f{{\boldsymbol \Theta} _{(\omega )}^{\mathsf p}} +\sum
_{k=1}^N \nm {x _k^Nf}{{{\boldsymbol \Theta} ^{\mathsf p}_{(\omega
)}}}\Big ),
\end{equation}
and the result follows in the case $N_1=0$ and $N_2=N$ from
this estimate.

\par

The general case follows now if combine \eqref{uppskattning1} with
\eqref{uppskattning2}, which proves (2). The proof is complete.
\end{proof}

\par

\begin{proof}[Proof of Proposition \ref{identities2}.]
The first embeddings in \eqref{embcoorb1} follows immediately from
Proposition \ref{p1.4AA}. Next we prove \eqref{embcoorb2}. Let $\ep
>0$ be chosen such that $s_2-2\ep >n(q+1)/q$, $\mathbf E_{a,\omega
_N}$ be as in Section \ref{sec2}, and set
$$
F_{a,\omega _N}(\xi ,\eta ,z) = \Big ( \iint _{\rr {2n}} \, \sup
_{y\in \rr n} \mathbf E_{a,\omega _N} (X,\xi ,\eta ,z)^{p}\, dxd\zeta
\Big )^{1/p}.
$$
Then H{\"o}lder's inequality gives
\begin{multline*}
\nm a{{\boldsymbol \Theta} _{(\omega _N)}^{\mathsf p}} =  \int _{\rr
n} \Big ( \iint _{\rr {2n}}  F_{a,\omega _N}(\xi ,\eta ,z) ^q\, d\xi
dz\Big )^{1/q}\, d\eta
\\[1ex]
= 
\int _{\rr n} \Big ( \iint _{\rr {2n}}  F_{a,\omega _{N+s_2}}(\xi
,\eta ,z)^q\eabs {\xi ,\eta ,z}^{-s_2q} \, d\xi dz\Big )^{1/q}\, d\eta
\\[1ex]
\le
\int _{\rr n} \Big ( \iint _{\rr {2n}}  F_{a,\omega _{N+s_2}}(\xi
,\eta ,z)^q\eabs {\xi ,z}^{-(2n+\ep )} \, d\xi dz\Big )^{1/q}\eabs
\eta ^{-(n+\ep )}\, d\eta
\\[1ex]
\le C\nm {F_{a,\omega _{N+s_2}}}{L^\infty} =C\nm a{{\boldsymbol
\Theta} _{(\omega _{N+s_2})}^{\mathsf p_2}},
\end{multline*}
where
$$
C=\Big ( \iint _{\rr {2n}} \eabs {\xi ,z}^{-(2n+\ep )} \, d\xi dz\Big
)^{1/q}\int _{\rr n} \eabs \eta ^{-(n+\ep )}\, d\eta <\infty .
$$
This proves the first inclusion in \eqref{embcoorb2}. The second
inclusion follows by similar arguments. The details are omitted.

\par

Next we prove \eqref{embcoorb3}. By Lemma \ref{lemma1identities2} it
follows that we may assume that $\omega =1$ and $N=0$. By Remark
\ref{p1.7} (2) we have
$$
{\boldsymbol \Theta}^{\mathsf {p_1}} \subseteq M^{\infty ,1}\subseteq
C\cap L^{\infty} .
$$
Furthermore, if $\chi \in \mathscr S(\rr {3n})$ is such that $\chi
(0)=(2\pi )^{-3n/2}$, then it follows by Fourier's inversion formula
that
$$
a(X) = \iiint _{\rr {3n}}V_\chi a(X,\xi ,\eta ,z)e^{i(\scal x\xi +\scal y\eta
+\scal \zeta z)}\, d\xi d\eta dz.
$$
Hence Minkowski's inequality gives
\begin{multline*}
\nm a{C ^{0, p}} =\Big ( \iint _{\rr {2n}}\, \big (\sup _{y\in \rr
n}|a(X)|\big )^p\, dxd\zeta \Big )^{1/p}
\\[1ex]
\le
\Big ( \iint _{\rr {2n}}\, \sup _{y\in \rr
n}\Big ( \iiint _{\rr {3n}}|V_\chi a(X,\xi ,\eta ,z)|\, d\xi d\eta dz
\Big )^p\, dxd\zeta \Big )^{1/p}
\\[1ex]
\le
\iiint _{\rr {3n}} \Big ( \iint _{\rr {2n}}\, \sup _{y\in \rr
n}|V_\chi a(X,\xi ,\eta ,z)|^p\, dxd\zeta \Big )^{1/p}\, d\xi d\eta dz
=\nm a{\boldsymbol \Theta ^{\mathsf p_1}}.
\end{multline*}
This proves the right embedding in \eqref{embcoorb3}.

\par

In order to prove the left embedding in \eqref{embcoorb3} we observe
that
$$
|V_\chi a(X,\xi ,\eta ,z)| \le (2\pi )^{-3n/2}\int |\chi
(X_1-X)a(X_1)|\, dX_1 = (2\pi )^{-3n/2} (|a|*|\check \chi | )(X),
$$
which together with Young's inequality give
\begin{multline*}
\nm a{\boldsymbol \Theta ^{\mathsf p_2}} =\sup _{\xi ,\eta ,z}\Big
(\iint _{\rr {2n}}\sup _{y\in \rr n}|V_\chi a(X,\xi ,\eta ,z)|^p\,
dxd\zeta \Big )^{1/p}
\\[1ex]
\le
C(\iint _{\rr {2n}}\sup _{y\in \rr n}(|a|*|\check \chi|)(X)^p\,
dxd\zeta \Big )^{1/p}
\\[1ex]
\le C\nm \chi {L^1}(\iint _{\rr {2n}}\sup _{y\in
\rr n}|a(X)|^p\, dxd\zeta \Big )^{1/p} = C\nm \chi {L^1}\nm a{C^{0,p}},
\end{multline*}
for some constant $C$. Hence if $\omega (X,\xi ,\eta ,z)=\eabs {\xi
,\eta ,z}^{-3n-1}$, then it follows from Lemma \ref{lemma1identities2}
that
\begin{multline*}
\nm a{{\boldsymbol \Theta} ^{\mathsf p_1}} \le C _1\nm {\omega
^{-1}(D)a}{{\boldsymbol \Theta} ^{\mathsf p_1}_{(\omega )}}
\le C _2\sum _{|\alpha |\le 3n+1}\nm {a^{(\alpha )}}{{\boldsymbol
\Theta}^{\mathsf p_1}_{(\omega )}}
\\[1ex]
\le C _2\sum _{|\alpha |\le 3n+1}\nm {a^{(\alpha )}}{\boldsymbol \Theta
^{\mathsf p_2}}\le C _3\sum _{|\alpha |\le 3n+1}\nm {a^{(\alpha
)}}{C^{0,p}} =C_3\nm a{C^{N,p}},
\end{multline*}
for some constants $C_1,\dots ,C_3$. This proves \eqref{embcoorb3} and
the result follows.
\end{proof}

\par

\begin{cor}\label{identities3}
Let $N$, $\omega _s$ and $\mathsf p$ be as in Proposition
\ref{identities2}. Then
$$
C_{(\omega )}^{\infty ,p} = \boldsymbol \cap _{N\ge 0}\Theta _{(\omega
_N)}^{\mathsf p}
$$
\end{cor}

\par

\begin{rem}\label{remidentities2}
Similar properties with similar motivations as those in Proposition {identities2}, Lemma \ref{lemma1identities2} and Corollary {identities3}, and their proofs, also holds when the $\Theta _{(\omega _N)}^{\mathsf p}$ spaces and $C^{N,p}_{(\omega )}$ spaces are replaced by the modulation space $M^{p,q}_{(\omega )}(\rr n)$ for $\omega \in \mathscr P(\rr {2n})$ and
$$
\sets {f\in \mathscr S'(\rr {n})}{f^{(\alpha )}\in M^{p,q}_{(\omega )}(\rr n)\, |\alpha |\le N}
$$
respectively. (Cf. \cite {To9}.)
\end{rem}

\par

Now we may combine Proposition \ref{identities2} with the results in Section \ref{sec2} to
obtain continuity properties for certain type of
Fourier integral operator when acting on modulation spaces. The following result is a
consequence of Theorem \ref{thm3.1C} and Proposition \ref{identities2}.

\par

\begin{thm}\label{thm3.1D}
Assume that $n_1=n_2=m=n$, $\omega \in \mathscr P(\rr {6n})$ and
$\widetilde \omega \in \mathscr P(\rr {3n})$ satisfy
$$
\omega (X,\xi ,\eta ,z)=\widetilde \omega (X)\eabs {\xi ,\eta ,z}^N
$$
for some constant $N$, and that $\chi$, $\omega _j$,
$v$ and $\fy$ for $j=0,1,2$ are the same as in Subsection
\ref{ssec2.1}. Also assume that
$p\in [1,\infty]$, $a\in C^{\infty ,p}_{(\widetilde \omega )}(\rr {3n})$, and
that $|\det (\fy ''_{y ,\zeta})|\ge \mathsf
d$ and \eqref{detphicond} hold for some $\mathsf d>0$. Then the following is true:
\begin{enumerate}
\item {\rm{(i)--(ii)}} in Subsection \ref{ssec2.2} holds;

\vrum 

\item $\op _\fy (a)\in \mathscr I_p(M^2_{(\omega _1)},M^2_{(\omega _2)})$.
\end{enumerate}
\end{thm}

\par

\end{document}